\documentclass[english]{amsart}
\usepackage[english]{babel}

\usepackage{hyperref}
\usepackage{verbatim}
\usepackage{amsfonts,amssymb,mathtools}
\usepackage{amsmath}
\usepackage{bbm}
\usepackage{hyperref}
\usepackage[T1]{fontenc}
\usepackage{mathtools}

\usepackage{graphicx,xcolor}
\definecolor{green}{RGB}{89,169,58}
\definecolor{red}{RGB}{224,61,42}
\definecolor{blue}{RGB}{63,78,181}

\graphicspath{{Figures/}}
\usepackage{subfigure,tikz}
\usetikzlibrary{arrows}

\usepackage[style=numeric,maxcitenames=2,maxbibnames=10,doi=false,isbn=false,url=false,natbib=true,giveninits=true,hyperref,bibencoding=utf8,backend=biber]{biblatex}
\AtEveryBibitem{%
  \clearfield{issn}%
  \clearlist{language}%
  }

\newcommand{\N}{\ensuremath{\mathbb{N}}}
\newcommand{\R}{\ensuremath{\mathbb{R}}}

\newcommand{\Z}{\ensuremath{\mathbb{Z}}}
\newcommand{\E}{\ensuremath{\mathbb{E}}}
\renewcommand{\P}{\ensuremath{\mathbb{P}}}

\newcommand{\ind}[1]{\ensuremath{\mathbbm{1}_{\left\{#1\right\}}}}
\newcommand{\diff}{\mathop{}\mathopen{}\mathrm{d}}
\newcommand{\cal}[1]{\ensuremath{\mathcal{#1}}}
\newcommand\croc[1]{\left\langle #1\right\rangle}
\newcommand\steq[1]{\stackrel{\text{\rm #1.}}{=}}
\newcommand{\I}[1]{\ensuremath{I_{[#1]}}}

\def\eps{\varepsilon}
\def\cadlag{c\`adl\`ag }

\newtheorem{proposition}{Proposition}
\newtheorem{definition}[proposition]{Definition}
\newtheorem{lemma}[proposition]{Lemma}
\newtheorem{theorem}[proposition]{Theorem}

\newtheorem{corollary}[proposition]{Corollary}

\title[k-Unary CRN]{Analysis of Stochastic Chemical Reaction Networks with a Hierarchy of Timescales}

\date{\today}
\author[L. Laurence]{Lucie Laurence${ }^1$}
\email{Lucie.Laurence@inria.fr}
\address[L. Laurence, Ph. Robert]{INRIA Paris, 48, rue Barrault, CS 61534, 75647 Paris Cedex, France}
\author[Ph. Robert]{Philippe Robert}
\email{Philippe.Robert@inria.fr}
\urladdr{http://www-rocq.inria.fr/who/Philippe.Robert}
\thanks{${}^1$Supported by PhD grant of ENS-PSL}

\begin{document}
\begin{abstract}
 We investigate a class of stochastic chemical reaction networks  with $n{\ge}1$ chemical species $S_1$, \ldots, $S_n$, and whose complexes are only of the form $k_iS_i$, $i{=}1$,\ldots, $n$, where $(k_i)$ are integers. The time evolution of these CRNs  is driven by the kinetics of the law of mass action. A scaling analysis is done when the rates of external arrivals of chemical species are proportional to a large scaling parameter $N$. A natural hierarchy of fast processes, a subset of the coordinates of $(X_i(t))$,  is determined by the values of the mapping $i{\mapsto}k_i$. We show that the scaled vector of coordinates $i$ such that $k_i{=}1$ and the scaled occupation measure of the other coordinates are converging in distribution to a deterministic limit as $N$ gets large. The proof of this result is obtained by establishing a functional equation for the limiting points of the occupation measure,  by an induction on the hierarchy of timescales and with relative entropy functions.
\end{abstract}
\maketitle

\vspace{-5mm}

\bigskip

\hrule

\vspace{-3mm}

\tableofcontents

\vspace{-1cm}

\hrule

\bigskip

\section{Introduction}\label{Intro}
A {\em stochastic chemical reaction network} (CRN) with $n$ chemical species is  described as a continuous time Markov process $(X_i(t))$ on a subset of $\N^n$. The $i$th component gives the number of molecules of chemical species $S_i$, $1{\le}i{\le}n$. Its dynamical behavior is given by a finite set of chemical reactions which add or remove simultaneously  a finite number of several chemical species. For example, the reaction
\begin{equation}\label{Rex}
k_1S_1{+}k_2S_{2}  \xrightharpoonup{\kappa} k_3S_{3}
\end{equation}
transforms $k_1$ molecules of $S_1$ and $k_2$ molecules of $S_2$ into $k_3$ molecules of $S_3$. The associated transition of this reaction for the Markov process is
\[
x{=}(x_i)\to x{+}k_3e_3{-}k_1e_1{-}k_2e_2,
\]
where $e_i$, $1{\le}i{\le}n$, is the $i$th unit vector of $\N^n$. The rate at which the reaction occurs is assumed to follow the {\em law of mass action},  for our example the rate is given by 
\begin{equation}\label{LMAex}
\kappa x_1^{(x_1)} x_2^{(x_2)}\steq{def}\kappa\frac{x_1!}{(x_1{-}k_1)!}\frac{x_2!}{(x_2{-}k_2)!},
\end{equation}
for some positive constant $\kappa$. See~Section~\ref{Model}. 

From a mathematical point of view, there are two important characteristics of stochastic models of CRNs described with Markov processes. 
\begin{enumerate}
\item {\sc Polynomial Reaction Rates.}\label{Polit}\\
When the coordinates $x_1$ and $x_2$ are large, the reaction rate~\eqref{LMAex} is of the order of $\kappa x_1{}^{k_1}x_2{}^{k_2}$. This implies that some reactions will be much more likely than others, and therefore will dominate the kinetics of the CRN, for a while at least. In this case, we will speak of fast processes for the coordinates involved in these reactions. There are many examples of such behavior. See~\citet{Agazzi2018}, \citet{BallKurtz}, \citet{Togashi_2001} and Sections~6, 7, 8 of~\citet{LR23} for example. This is a major  feature of CRNs from a technical point of view. In such a case, a CRN can be described as driven by a set of interacting fast processes leading to an investigation of possible stochastic averaging principles or even more complex multi-timescales behaviors. See Section~\ref{22LitSec}.

\item {\sc Boundary Behavior.}\\
This feature is due to a constraint on the state space rather than a property related to the order of magnitude of transition rates. In state $x{=}(x_i){\in}\N^n$, Reaction~\eqref{Rex} occurs only if $x_1{\ge}k_1$ and $x_2{\ge}k_2$. Mathematically, this is a kind of discontinuity of the kinetics of the CRN.
This constraint on the state space is at the origin of  complex behaviors of CRNs. In the CRN of example \eqref{Rex}, if we assume that $X_1(0){=}N$ is large and that the process $(X_2(t))$ remains in a neighborhood of $0$, then the process $(X_1(t))$ will decrease only during the excursions of $(X_2(t))$ above $k_2$. This can be even more complicated if the dynamic of $(X_2(t))$ depends, via other chemical reactions, on $(X_3(t))$ for example.  For example of such  complex behaviors, see  Section~8 of~\cite{LR23}  and~\citet{LR24}. 
\end{enumerate}

\subsection{$\mathbf{k}$-Unary Chemical Reaction Networks}
We now describe the class of CRNs analyzed in our paper. As it will be seen boundary behaviors  play only a marginal role in the time evolution of these networks. The characteristic~\eqref{Polit} on the polynomial growth is the key feature. 

The parameters of the kinetics of these networks are given by the coefficients of a matrix $R_{\kappa}{=}(\kappa_{ij}, 0{\le}i, j{\le}n){\in}\R_+^{n+1}{\times}\R_+^{n+1}$, and a vector $(k_i){\in}\N^n$ of integers.  The only chemical reactions for this class of CRNs are as follows, for $1{\le}i{\ne}j{\le}n$,
\[
k_iS_i \xrightharpoonup{\kappa_{ij}} k_jS_{j},\quad k_iS_i \xrightharpoonup{\kappa_{i0}}\emptyset,\quad  \emptyset \xrightharpoonup{\kappa_{0i}}k_iS_i,
\]
provided that, respectively,  $\kappa_{ij}{>}0$, $\kappa_{i0}{>}0$, or $\kappa_{0i}{>}0$.  The second reaction, resp. last reaction,  is the spontaneous destruction, resp. creation, of $k_i$ molecules of chemical species $S_i$. The symbol $\emptyset$ is the source/sink for chemical species.

For $1{\le}i{\le}n$,  $k_iS_i$ is the only {\em complex} involving  the chemical species $S_i$ and  the time evolution of the $i$th coordinate is a jump process whose jumps are ${\pm}k_i$.   In state $x{=}(x_k)$, for $i\in\{1,\ldots,n\}$, the $i$th coordinate decreases at a rate proportional to $x_i^{(k_i)}$ and, for $1{\le}j{\le}n$,   $\kappa_{ij}x_i^{(k_i)}$  is the rate at which  $k_i$ molecules of $S_i$ are transformed into  $k_j$ molecules of $S_j$. These are the kinetics of the law of mass action.  See Section~\ref{Model}. 

This class of CRNs has in fact an invariant distribution, see Relation~\eqref{invMeas} of Section~\ref{Def0},  given by a product of Poisson distributions. If this is satisfactory, it should be noted that there are many very different Markov processes with this property, see~\cite{Kelly}. It does not give much insight on the transient characteristics of the CRNs, in particular on the impact of its different timescales of this CRNs, if any.

A scaling approach is proposed to investigate the dynamical behavior of these networks. We quickly review several scalings already used in the literature of stochastic CRNs. 

\subsection{Scaling Methods for Chemical Reaction Networks}\label{ScSecC}
We denote by $N$ the scaling parameter. 
\begin{enumerate}
\item  Classical Scaling.\\
For this scaling the reaction rate $\kappa_r$ of a chemical reaction $r$,   is scaled in $N$, as $\kappa_r/N^\gamma_r$ for some $\gamma_r{\ge}0$, so that if all coordinates of the associated Markov process $(X(t)){=}(X_i(t))$ are of the order of $N$, then the transition rate of any jump of the process  is of the order of $N$. See~\citet{Mozgunov} or Proposition~2 of ~\citet{LR23}for example. In this case, under appropriate conditions, it can be shown that the process $(X_i^N(t)/N)$ is converging in distribution to the solution of an ODE whose stability properties have been investigated in the literature of deterministic CRNs. See~\citet{Feinberg1972} and~\citet{Horn1972} for example. 

This scaling has the effect of  somewhat equalizing the kinetics of the CRNs. There cannot be a subset of chemical reactions dominating at some moment for a while, since all transition rates are of the order of $N$.

Kurtz and co-authors have also investigated several examples of CRNs with related scaling methods. In this approach, some reaction rates  may be sped-up with some power of the scaling parameter and the state variables are scaled accordingly. There is no requirement that all reactions have the same order of magnitude. The initial motivation was of fitting the parameters of these scaling models with  biological data obtained from experiments.  See for example~\citet{BallKurtz}, \citet{KangKurtz}, and~\citet{Kim2017} where, for several examples of CRNs,  the choice of convenient scalings of reaction rates is investigated  and several limit theorems are derived. 

\item Scaling with the norm of the initial state.\\
In this approach the reaction rates $\kappa_r$ are fixed so that the topology of the CRN is preserved by the scaling. The scaling parameter for the Markov process $(X(t))$ is $N{=}\|X(0)\|$. The approach consists in describing, via possibly functional limit theorems, how the sample path of the state of the CRN returns to a neighborhood of the origin. This is a natural way to investigate positive recurrence properties of the CRNs but, more importantly, it can provide insight into  transient characteristics of CRNs. Up to now there are few results in the literature in this domain, see~\citet{AgazziDembo2} and~\cite{AgazziDembo}, ~\citet{Mielke}, and~\citet{Sweeney}. For the scaling with $\|X(0)\|$, see~\citet{LR23} and references therein. 
\end{enumerate}

\subsection*{Scaling External Input Rates}
The scaling investigated in this paper is as follows. For all $i{\in}\{1,\ldots,n\}$ such that $\kappa_{0i}{>}0$, the creation of chemical species $S_i$ is scaled by $N$,   it becomes
\[
 \emptyset \xrightharpoonup{N\kappa_{0i}}k_iS_i.
 \]
 The other reaction rates do not change. Rather than starting from a ``large'' initial state, this scaling regime  assume heavy traffic conditions at the entrance of the CRNs. A natural question in this setting is of establishing a limit theorem on the orders of magnitude in $N$ of the coordinates of $(X_N(t)){=}(X_i^N(t))$. This scaling has already been considered in~\citet{Togashi_2007} for CRNs and in~\citet{BallKurtz}, and probably in many other examples. A related scaling has also been used to investigate the transient behavior of Markov processes for stochastic models of large communication networks in~\citet{Kelly1986}. See also~\citet{Kelly} for a survey.

A basic example of such a situation is the $k$-unary CRN with one chemical species,
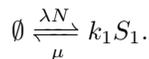
\[
	\emptyset \mathrel{\mathop{\xrightleftharpoons[\mu]{\lambda N}}} k_1 S_1.
\]
It can be easily seen that, under convenient initial conditions, the scaled process
\begin{equation}\label{EhrL}
\left(\frac{X_1^N\left(t/N^{(1{-}1/k_1)}\right)}{N^{1/k_1}}\right)
\end{equation}
converges in distribution to a non-trivial deterministic function, the solution of an ODE.  See Proposition~\ref{EhrProp}.

\subsection*{A Hierarchy of Timescales}
We come back to our CRNs  under the heavy traffic assumptions, i.e. with all external input rates scaled by $N$. Heuristically, if there is a kind of equilibrium of flows in the network at some moment, due to the external inputs of the order of $N$, the input flow through each node  should be also of the same order of $N$.

The case of the CRN with a single node suggests then that the  state variable of the $i$th node $(X^N_i(t))$, $1{\le}i{\le}n$,  should be of the order of $N^{1/k_i}$. The convergence result for the process~\eqref{EhrL} indicates that the ``natural'' timescale of $(X_i(t))$ should be $(t/N^{(1{-}1/k_i)})$. In particular, this implies that, at the ``normal''  timescale $(t)$,  all coordinates  $(X_i(t))$ whose index $i{\in}\{1,\ldots,n\}$ is such that $k_i{\ge}2$,  are fast processes. The CRN exhibits  in fact a hierarchy of timescales: The process associated to $(X_j(t))$ is faster than the process $(X_i(t))$ provided that $k_j{>}k_i$. A limit theorem to establish the convergence  of the scaled process
\begin{equation}\label{ScPrI}
\left(\frac{X_i^N(t)}{N^{1/k_i}}\right)
\end{equation}
has to handle this multi-timescales feature and also the interactions with the other coordinates. 

\subsection{Literature}\label{22LitSec}
A classical way of investigating multi-timescales processes is via the proof of an averaging principles. 
Averaging principles have already  been used in various situations to study chemical reaction networks (CRNs). In most of cases, it involves two timescales: there are  a fast process and a slow process. Early works on the proof of averaging principles are due to Has'minski\v{\i}. See~\citet{Khasminskii0,Khasminskii1}. Chapter~7 of~\citet{Freidlin} considers these questions in terms of the convergence of Cesaro averages of the fast component. \citet{Papanicolaou} has introduced a  stochastic calculus approach to these problems, mainly for diffusion processes. \citet{Kurtz1992} has extended this approach to jump processes. For CRNs, there are numerous proofs of averaging principles in such a setting: \citet{BallKurtz}, \citet{KangKurtz},  \citet{Kim2017}, \citet{LR23,LR24}, \ldots

With more than two timescales, limit theorems  in a stochastic framework are quite scarce in the literature. 
A model with three timescales is investigated in~\citet{KangKurtzPop}, and a functional central limit result is established.  In this reference, it is  assumed that the first order is deterministic.  To handle the two fast timescales, several assumptions on uniform convergence of infinitesimal generators  on compact subsets of the state space are introduced. Large deviations results are derived with similar assumptions in~\citet{Popovic2019}.  It does not seem that such an approach can be used in our case. 

A stochastic model of a CRN with three timescales is analyzed in~\citet{FRZ23}. The limiting behavior of the occupation measure of the  processes associated to the two fast timescales is investigated. The main difficulty is of identifying the possible limits. A technical result on conditional probabilities is the major ingredient to solve this problem. This method do not seem to be possible for our CRN, mainly because there are too many fast timescales a priori, so that an analogous result on conditional probabilities is not clear.

\subsection{Outline of the Paper}
The goal of this paper is of establishing a limit theorem for the convergence in distribution of the scaled process defined by Relation~\eqref{ScPrI}~:
\begin{itemize}
\item For the occupation measure of the coordinates of the Markov process whose indices $i{\in}\{1,\ldots,n\}$ are such that $k_i{\ge}2$;
\item For the vector of the other components, i.e. indices  $i{\in}\{1,\ldots,n\}$ with $k_i{=}1$,  for the uniform topology.
\end{itemize}
See Theorem~\ref{Theorem} for the full statement. The proof of this result is done in several steps. 
\begin{enumerate}
\item Technical estimates of the ``basic'' model of a $k$-unary CRN with one chemical species in Section~\ref{Ehrenfest};
\item Tightness results for the occupation measure by using (a) and linear algebra arguments in Section~\ref{Bounds};
\item Identification of the limit of the sequence of occupation measures. This is done first by establishing a functional equation for some marginals of the possible limiting points, Relation~\eqref{eqLimitMeas} of Proposition~\ref{propEqLimit}, and then by induction on the hierarchy of timescales starting from the fastest timescale. Relative entropy functions associated to each timescale and convexity arguments are the main ingredients of the proofs. In Section~\ref{FastSec} when all $k_i$, $i{=}1,\ldots,n$, are greater than $2$, and Section~\ref{1GenSec} for the general case. 
\end{enumerate}

\section{Stochastic Model}\label{Model}
We introduce the formal definitions and notations used throughout the paper. 
\subsection{The class of $k$-unary Chemical Reaction Networks}\label{kDef}
\begin{definition}[$k$-unary CRN]\label{CondConnect}
The components of  a \emph{$k$-unary  chemical reaction network}  are :
\begin{enumerate}
\item A set of $n$ distinct {\em chemical species} ${\cal S}{=}\{S_1,\ldots,S_n\}$. The set ${\cal S}$ is also identified to $\{1,\ldots,n\}$ and  $\emptyset$ is the source/sink for chemical species, it is associated to index $i{=}0$ in general;
\item {\em Complexes} ${\cal C}$  are of the form $k_iS_i$,  $i{=}1$, \ldots, $n$, we will have the convention $k_0{=}0$.  Each species is present in exactly one complex. 
\item The {\em rates of chemical reactions} are associated to a $Q$-matrix $R_{\kappa}{=}(\kappa_{ij},i,j{\in}I)$ of a jump Markov process on $I{=}\{0,\ldots,n\}$ in the following way: If $i$, $j{\in}I$ are such that $\kappa_{ij}{>}0$, then there is the reaction 
\[
\begin{cases}
  k_iS_i\xrightharpoonup{\kappa_{ij}} k_jS_j& \text{ if } i{\ne}0;\\
  \emptyset\xrightharpoonup{\kappa_{0j}N} k_jS_j& \text{ if } i{=}0,
\end{cases}
\]
where $N$ is the {\em scaling parameter}. These are the only possible reactions.
\end{enumerate}
\end{definition}
Note that the process associated to the $Q$-matrix $R_{\kappa}$ \emph{is not} the process  describing the time evolution of the CRN,  it is a jump  process on the finite set $I$. 
The state of the CRN is given by $(X_N(t)){=}(X_{N,i}(t))$, a Markov process with values in $\N^n$.  
Since, for $i{\in}\{1,\ldots,n\}$, the sizes of jumps of the number of copies of chemical species $i$ are either $\pm k_i$, a natural state space  for this process is
\begin{equation}\label{StSp}
{\cal S}_a{=}\left\{x{=}(x_i){=}(a_1{+}m_1k_1,a_2{+}m_2k_2, \ldots ,a_n{+}m_nk_n): (m_i){\in}\N^n\right\}, 
\end{equation}
for any $a{\in}\{0,\ldots,k_1{-}1\}{\times}\{0,\ldots,k_2{-}1\}{\times}\cdots {\times}\{0,\ldots,k_n{-}1\}$.

The kinetics of the system are driven by   \emph{the law of mass action}, see~\citet{Voit2015}, \citet{lund1965guldberg} for surveys on the law of mass action and the historical reference \citet{guldberg1864studies}. The associated transitions are thus given by, for $x{\in}{\cal S}_a$, $i$, $j{\in}I$, $i{\ne}0$, 
\[
	x{=}(x_i)\rightarrow x+
	\begin{cases}
		k_je_j-k_ie_i, \quad &\text{at rate} \quad \kappa_{ij}x_i^{(k_i)}\\
		k_ie_i,\quad &\phantom{at} ``\phantom{rate} \quad \kappa_{0i}N\\
		-k_ie_i,\quad &\phantom{at} ``\phantom{rate} \quad \kappa_{i0}x_i^{(k_i)}.
	\end{cases}
\]
where $e_i$ is the $i$th unit vector of $\N^{n}$ and, for $y$, $k{\in}\N$,
\begin{equation}\label{Powk}
	y^{(k)}= \frac{y!}{(y{-}k)!},
\end{equation}
if $y{\ge}k$ and $y^{(k)}{=}0$ otherwise.

Such CRNs have a \emph{fast input}, in the sense that the rates of creations of chemical species  are proportional to a (large)  scaling factor $N$, and these are the only chemical reactions which are sped-up. 

\begin{figure}[ht]
  \centerline{
    \begin{tikzpicture}[->,node distance=2cm]
			\node (A) [below] {$\emptyset$};
			\node	(E) [right of=A] {$3S_1$};
			\node (F) [right of =E] {$3S_2$};
			\node (C) [below of=E] {$2S_3$};
			\node (B) [below of=F] {$S_4$};
			\draw[-left to] (A)  -- node[rotate=0,above] {$\kappa_{01}N$} (E);
			\draw[-left to] (E)  -- node[rotate=0,above] {$\kappa_{12}$} (F);
			\draw[-left to] (E)  -- node[rotate=0,left] {$\kappa_{13}$} (C);
			\draw[-left to] (F)  -- node[rotate=0,right] {$\kappa_{23}$} (C);
			\draw[-left to] (C)  -- node[rotate=0, left] {$\kappa_{30}$} (A);
			\draw[-left to] (F)  -- node[rotate=0, right] {$\kappa_{24}$} (B);
			\draw[-left to] (C)  -- node[rotate=0, below] {$\kappa_{34}$} (B);
			\draw[-left to] (B)  -- node[rotate=0, above] {$\kappa_{43}$} (C);
  \end{tikzpicture}}
\caption{An example of a $k${-}unary CRN}\label{FourSpecies}
\end{figure}
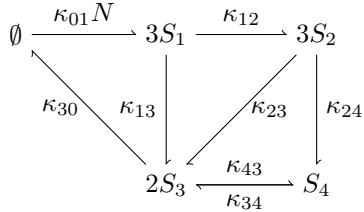

\subsection{Relations with Some Queueing Networks}
A $k$-unary network can be  related to several queueing systems.
\begin{enumerate}
\item When there is one chemical species, this is a generalized $M/M/\infty$ queue. See Section~\ref{Ehrenfest}.
\item Queueing networks referred to as {\em Jackson Networks}. They can be described simply as follows. 
\begin{itemize}
\item There are $n$ sites for the location of jobs. If $x{=}(x_j){\in}\N^n$, for $1{\le}j{\le}n$, $x_j$ denotes the number of jobs at the site $j$. 
\item One of the jobs at site $i$ leaves at rate $\mu_{ij}{>}0$ to go to site $j{\in}\{1,\ldots,n\}$, or leave the network at rate $\mu_{i0}$
\item External jobs arrive  at the site $i$ at rate $\mu_{0i}{\ge}0$.
\end{itemize}
The main difference with our CRNs is that the $i$th coordinate, $i{\in}\{1,\ldots,n\}$ decreases at a fixed rate if it is not $0$, instead of a rate proportional to $x_i^{(k_i)}$ for a $k$-unary CRN. There is a scaling result with the norm of the initial state for these networks in~\citet{Chen1991},    see Section~\ref{ScSecC}. The scaling results are quite different, there is only one timescale for Jackson  networks. Nevertheless,   as for our CRNs, a linear system plays an important role in the limit theorems associated to these Markov processes. See Relation~\eqref{system2} in Proposition~\ref{InvK} for $k$-unary CRNs and, for Jackson networks, see Proposition~9.6 of~\citet{Robert}. 
\end{enumerate}

\subsection{Notations}\label{Not}
Throughout the paper, the following notations will be used. For a subset $A$ of $\R$, we denote $A^*{=}A{\setminus}\{0\}$ and, for $p$, $q{\in}\N^*$, $p{\leq}q$, 
\begin{equation}
  	I_{[p, q]}\steq{def} \{0\}\cup \{i\ge 1: p{\leq} k_i{\leq} q\},
\end{equation}
with the convention that $\I{p}{=}\I{p,p}$, and $\I{p+}{=}\I{p, {+}\infty}$, so that $I{=}\I{1+}$.

If $x{\in}\R_+^{I^*}$ and $1{\le}p{\le}q$, we define $x_{[p,q]}{=}(x_i, i{\in}\I{p,q}^*)$ and $x$ will also be represented as
$x{=}(x_{[k_1]},x_{[k_2]},\ldots, x_{[k_n]})$ or $x{=}(x_{[1,q-1]},x_{[q+]})$, for $q{\ge}2$, provided that $\I{1,q-1}^*$ and $\I{q+}^*$ are non-empty.
Similarly, if $\pi$ is a probability distribution on on $\R_+^{I^*}$, $\pi^{[p,q]}$ is the distribution on $\R_+^{\I{p,q}^*}$ of marginals of $\pi$ for the coordinates whose index is in $\I{p,q}^*$, i.e. the image of $\pi$ by the mapping $x{\mapsto}x_{[p,q]}$. 

For any subset $A$ of $I$, we  denote by $\Omega(A)$ the set of irreducible $Q$-matrices $(x_{ij},i,j{\in}A)$ for the state space $A$. We will assume in this paper that $R_{\kappa}{\in}\Omega(I)$,  with a slight abuse of notation we will also write $\kappa{\in}\Omega(I)$. An $A{\times}A$-matrix refers to a $|A|{\times}|A|$-matrix, with $|A|$ the cardinality of $A$.
For $i{\in}I^*$,  we define
	\begin{equation}\label{KappaBullet}
		\kappa_{i}^+{=}\kappa_{i0}{+}\sum_{j\in I^*\setminus\{i\}} \kappa_{ij}.  
	\end{equation}
We now introduce a natural distance $(d(i))$  from the origin (the complex $\emptyset$) on the graph of the CRN. 
\begin{definition}\label{definitiond}
We set  $d(0){=}0$ and, for  $1{\leq}i{\leq}n$, 
	\[
 d(i)=\min \left\{k{\ge} 1: \exists i_1,\ldots, i_{k-1}{\in}I^*, \kappa_{0,i_1}{\cdot}\prod_{p=1}^{k-2}\kappa_{i_pi_{p+1}}{\cdot}\kappa_{i_{k-1}i}{>}0\right\}. 
	\]
\end{definition}
A real-valued function $(x(t))$ on $(\R_+)^{I^*}$ is \cadlag if it is right continuous and it has left-limits everywhere on $\R_+^*$, in this case, for $t{>}0$, $x(t{-})$ denotes the left limit of $(x(t))$ at $t{>}0$. If $H$ is a subset of $\R^d$, for $d{\ge}1$, we denote by ${\cal B}(H)$ the set of Borelian subset of $H$, $\cal{C}_c(H)$ the set of continuous functions on $H$ with compact support on $H$ and  $\cal{C}^2_c(H)$ the subset of class ${\cal C}_2$-functions and the set on Borelian probability distributions on $H$ is denoted as ${\cal P}(H)$. 

The paper convergence in distribution of a sequence of jump processes $(U_N(t))$ in $\R^d$ to a process $(U(t))$ is understood with respect to the topology of  uniform convergence on compact sets for \cadlag functions. See Chapters~2 and~3 of~\citet{Billingsley} for example. The convergence in distribution of the associated occupation measures is the convergence in distribution of the sequence of random measures $(\mu_N)$ on $\R_+^d$, defined by, for $f{\in}{\cal C}_c(\R_+^d)$, 
\[
\croc{\mu_N,f}=\int_0^Tf(s,U_N(s))\diff s.
\]
See~\citet{Dawson} for the technical aspects related to measure valued processes. 

\subsection{Stochastic Differential Equations}
We will express the time evolution of the $k$-unary CRN, as a \cadlag process  $(X_N(t)){=}(X_i^N(t),i{=}1,\ldots,n)$,  solution of the following stochastic differential equation (SDE). See~\citet{LR23}. 
For $i{\in} I^*$, $t{\geq}0$, 
\begin{multline}\label{SDE}
	\diff X_i^N(t)= k_i\cal{P}_{0i}\left( (0, \kappa_{0i}N), \diff t\right)+ \sum_{j\in I^*\setminus\{i\}} k_i \cal{P}_{ji}\left(\left( 0, \kappa_{ji}(X_j^N(t{-}))^{(k_j)}\right), \diff t\right)\\
	-\sum_{j\in I\setminus\{i\}} k_i \cal{P}_{ij}\left(\left( 0, \kappa_{ij}(X_i^N(t{-}))^{(k_i)}\right), \diff t\right).
\end{multline}
where $\cal{P}_{ij}$, $i,j\in I$ is a family of independent Poisson point processes on $\R_+^2$ with intensity measure the Lebesgue measure on $\R_+^2$. See~\citet{Kingman}.
If ${\cal P}$ is a positive Borelian measure on $\R_+^2$, and $A{\in}{\cal B}(\R_+)$ is a Borelian subset of $\R_+$, we use the following notation,
\begin{equation}\label{ConvBorel}
	{\cal P}(A, \diff t)= \int_{x{\in}\R_+} \ind{x\in A} {\cal P}(\diff x, \diff t). 
\end{equation}
The martingale, stopping time properties will refer to the smallest filtration $({\cal F}_t)$ satisfying the usual hypotheses and such that 
\[
\left\{\cal{P}_{ij}(A{\times}[0,s]): i{\in}I, j{\in}I\setminus\{i\}, A{\in}{\cal B}(\R_+), s\le t\right\} \subset {\cal F}_t,\quad \forall t{\ge}0. 
\]

\subsection{Invariant Distribution with Product Form Representation}\label{Def0}
In the language of chemical reaction networks, a  $k$-unary CRN is  \emph{weakly reversible} with \emph{one linkage class} and its {\em deficiency} is $0$. See~\citet{Feinberg} for the general definitions for CRNs.

\subsection*{The Deterministic CRN}\ \\ 
In a deterministic setting, a dynamical system $(u_N(t)){=}(u_{i}^N(t))$ on $\R_+^n$ is associated to this CRN
\begin{equation}\label{DetDS}
\frac{\dot{u}_{i}^N(t)}{k_i}=N\kappa_{0i}+\sum_{j\in I^*\setminus\{i\}} \left(u_{j}^N(t)\right)^{k_j}\kappa_{ji}  - (u_{i}^N(t))^{k_i} \sum_{j\in I\setminus\{i\}} \kappa_{ij}, \quad i{\in}I^*.  
\end{equation}
Classical results of~\citet{Feinberg1972} and~\citet{Horn1972} show that, in this case,  $(u_N(t))$ has a unique equilibrium point $\gamma_{N}{=}(N^{1/k_i}u_{i})$ which is locally stable, where $u_\infty{=}(u_{i})$ is the unique positive solution of the system of equations,
\begin{equation}\label{systemLimit}
\kappa_{0i} +\sum_{j\in I^*\setminus\{i\}} u_{j}^{k_j} \kappa_{ji} = u_{i}^{k_i}\sum_{j\in I\setminus\{i\}}\kappa_{ij},\quad  i{\in}I^*. 
\end{equation}
See Proposition~\ref{InvK} of Section~\ref {LAsec} and~\citet{Feinberg} for a general presentation of these dynamical systems.
\subsection*{The invariant Measure}\ \\ 
For  $a{=}(a_i){\in}\N^n$, with $a_i{\in}\{0, \ldots, k_i{-}1\}$ for all $1{\le}i{\le}n$, the Markov process $(X_N(t))$ is irreducible on the set ${\cal S}_a$ defined by Relation~\eqref{StSp}.  \citet{Anderson2010} shows that the invariant distribution of $(X_N(t))$ on ${\cal S}_a$ is given by 
\begin{equation}\label{invMeas}
	\nu_a(x)= \frac{1}{Z_a}\prod_{i=1}^n \frac{(\gamma_{j,N})^{x_i}}{x_i!}, x{\in}{\cal S}_a,
\end{equation} 
where $Z_a$ is the normalization constant,
\[
Z_a= \sum_{k{=}(k_i)\in \N^n} \prod_{i=1}^n \frac{(\gamma_{i,N})^{a_i+p_ik_i}}{(a_i{+}p_ik_i)!}. 
\]
and $\gamma_N{=}(\gamma_{i,N}){=}(N^{1/k_i} u_i)$, where $(u_i)$ is the solution of the system~\eqref{systemLimit}. 

\subsection{Timescales}\label{HiTS}
When $n{=}1$, the $k$-unary CRN is 
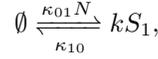
\[
	\emptyset \mathrel{\mathop{\xrightleftharpoons[\kappa_{10}]{\kappa_{01} N}}} k S_1,
  \]
        in state $x$, the instantaneous mean drift of $X_N$ is $k(\kappa_{01}N{-}\kappa_{10}x^{(k)})$. In view of Relation~\eqref{Powk}, to have a non-trivial time evolution  when $N$ is large,  this suggests that $x$ should be of the order of $N^{1/k}$. It is not difficult  to show that, provided that the sequence $(X_1^N(0)/N^{1/k})$ converges, then the sequence of processes 
        \[
        \left(\frac{X_1^N(t/N^{1{-}1/k})}{N^{1/k}}\right)
        \]
        is converging in distribution to $(x_1(t))$ the solution of the ODE
        \[
        \dot{x}_1(t)=k(\kappa_{01}{-}\kappa_{10}x_1(t)^k),\quad t{\ge}0.
        \]
        See  Section~\ref{Ehrenfest}. The natural timescale of the process $(X_1^N(t)/N^{1/k})$  is $(t//N^{1{-}1/k})$. If $k{\ge}2$,  $(X_1^N(t)/N^{1/k})$ is  then a {\em fast process}, and when $k{=}1$, $(X_1^N(t)/N)$ can be seen as  a {\em slow process}. 

 For our general $k$-unary CRN, fast and slow processes define a partition of the set of indices $i{\in}\{1,\ldots,n\}$ based on the fact that $k_i{=}1$ or $k_i{\ge}2$, i.e.  $I^*{=}\I{1}^*{\cup}\I{2+}^*$. In the same way,  if $i$, $j{\in}I^*$, is such that $k_i{>}k_j$, then the process $(X_i^N(t)/N^{1/k_i})$ is ``faster'' than the process $(X_j^N(t)/N^{1/k_j})$. This leads to a classification of chemical species according to their natural timescales, i.e. according to the non-decreasing sequence $(k_i)$.   This hierarchy plays an important role in the proofs of convergence in distribution of this paper. 

 \subsection{The Convergence Result}
With the above remark, the set $I_{2+}^*$ is the set of indices of fast processes, the asymptotic evolution of $({X}^N_i(t), i{\in} \I{2+}^*)$ is described only in terms of its {\em occupation measure}. For $I_{1}^*$, the set of indices associated to slow processes,  this is the  convergence in distribution of the sequence of processes $(\overline{X}^N_i(t), i{\in}\I{1}^*)$. 
\begin{definition}
 \begin{enumerate}
\item[]
\item The scaled process $(\overline{X}^N(t))$, is defined for $N{\geq}1$ as 
\begin{equation}\label{ScaledProcess}
	\left(\overline{X}_N(t)\right)=\left(\overline{X}_i^N(t)\right)= \left(\frac{X^N_i(t)}{N^{1/k_i}}\right). 
\end{equation}
The initial state  $X_N(0){=}x_N{=}(x_i^N){\in}\N^n$ of the process $(X_N(t))$ is assumed to satisfy the relation 
\begin{equation}\label{CondIni}
	\lim_{N\rightarrow +\infty}\left(\frac{x_i^N}{N^{1/k_i}}\right)= \left(\alpha_i\right){\in}\left(\R_+^*\right)^n. 
\end{equation}
\item The \emph{occupation measure} $\Lambda_N$ is the random measure on $\R_+{\times}\R_+^{I^*}$
defined by, for $g{\in}\cal{C}_c(\R_+{\times}(\R_+^*)^{I^*})$, 
\begin{equation}\label{OccMeas}
	\croc{\Lambda_N, g} = \int_{\R_+} g\left(u, \left(\overline{X}_i^N(u), i{\in}{I^*}\right)\right)\diff u. 
\end{equation}
        \end{enumerate}
\end{definition}
The main result of the paper is the following  theorem. 
\begin{theorem}\label{Theorem}
If  $(X_N(t))$ is  the solution of SDE~\eqref{SDE} whose initial condition satisfies Condition~\eqref{CondIni},
then,  for the convergence in distribution,
	\begin{equation}
		\lim_{N\rightarrow +\infty} \left( \left(\overline{X}^N_i(t), i{\in}{\I{1}^*}\right), \Lambda_N\right) = \left(\left(x_i(t),i{\in}{\I{1}^*}\right), \Lambda_\infty\right),
	\end{equation}
	where  $(\overline{X}_N(t))$ and occupation measure $\Lambda_N$ are defined respectively by Relations~\eqref{ScaledProcess} and~\eqref{OccMeas}, with, for  $g\in \cal{C}_c\left(\R_+{\times} (\R_+^*)^{I^*}\right)$, 
		\begin{equation}
			\croc{\Lambda_\infty, g} =\int_{\R_{+}} g\left(s, \left(x(s),\left(L_i(x(s)),{i{\in}\I{2+}^*}\right)\right)\right)\diff s,
		\end{equation}
where:
        \begin{enumerate}
        \item If $y{\in}(\R_+^*)^{\I{1}^*}$, $L(y){=}(L_i(y),i{\in}\I{2+}^*)$  is the unique solution of the system
		\begin{equation}\label{LinearSysEll}
			\kappa_{0i}+\sum_{j\in \I{1}^*} y_j\kappa_{ji} +\sum_{j\in \I{2+}^*\setminus\{i\}} L_j(y)^{k_j}\kappa_{ji}  = L_i(y)^{k_i} \sum_{j\in I\setminus\{i\}} \kappa_{ij}, \quad i{\in}\I{2+}^*;
		\end{equation}
         \item The function  $(x(t)){=}(x_i(t), i{\in}{\I{1}^*})$ is the unique solution of the set of ODEs, 
		\begin{multline}\label{ODE}
			\dot{x}_i(t)= \kappa_{0i}+\sum_{j\in \I{1}^*\setminus\{i\}} x_j(t)\kappa_{ji}\\ +\sum_{j\in \I{2+}^*\setminus\{i\}} L_j(x(t))^{k_j}\kappa_{ji}  - x_i(t) \sum_{j\in I\setminus\{i\}} \kappa_{ij}, \quad i{\in} \I{1}^*,
                \end{multline}
                with initial point $(\alpha_i, i{\in}{\I{1}^*})$. 
        \end{enumerate}
\end{theorem}
Not that the coordinates of the scaled vector $(\overline{X}_N(t))$  with indices in $\I{1}$ also appear in $\Lambda_N$ even if there is a much stronger result for the convergence in distribution for them. This is only to have simpler expressions.

\section{A Generalized $M/M/\infty$ Queue}\label{Ehrenfest}
In this section, we will study the simplest form of $k$-unary CRN, a CRN with only one species,
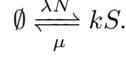
\[
	\emptyset \mathrel{\mathop{\xrightleftharpoons[\mu]{\lambda N}}} k S.
\]
The process $(X_N(t))$  is a birth and death process with the transition rates, for $x{\ge}0$,
\begin{equation}\label{QEhr}
x\longrightarrow x{+}
		\begin{cases}
			k&\text{ at rate } \lambda N,\\
			{-}k& \phantom{at } `` \phantom{arate } \mu x^{(k)}. 
\end{cases}
\end{equation}
When $k{=}1$, $(X_N(t))$ is the Markov process of the $M/M/\infty$ queue, with arrival rate $\lambda N$, and departure rate $\mu$. It is a basic model in the study of stochastic chemical reaction networks. See~\citet{LR23} and Chapter~6 of \citet{Robert} for a general presentation.

We start with a simple scaling result. 
\begin{proposition}\label{EhrProp}
	If the initial condition $x_n$ of the Markov process $(X_N(t))$ is such that 
        \[
	\lim_{N\rightarrow +\infty} \frac{x_n}{\sqrt[k]{N}} = \alpha,
        \]
        then, for the convergence in distribution, the relation
	\[
		\lim_{N\rightarrow +\infty} \left( \frac{1}{\sqrt[k]{N}}{X}_N\left(t/N^{1{-}1/k}\right), t\geq 0\right) = \left(x(t), t\geq 0\right),
	\]
holds, where $(x(t))$ is the solution of the ODE $\dot{x}(t){=}\lambda{-}\mu x(t)^k$, with $x(0){=}\alpha$. 
\end{proposition}
\begin{proof}
This is done with straightforward stochastic calculus. The SDE~\eqref{SDE} is in this case 
  \begin{equation}\label{Eeqd}
  \diff X_N(t)=k{\cal P}_{01}((0,\lambda N),\diff t){-}k{\cal P}_{10}((0,\mu X_N(t{-})^{(k)}),\diff t),
  \end{equation}
  by integrating this relation, we obtain that, for $t{\ge}0$, 
  \begin{equation}\label{Eeq}
  Y_N(t)\steq{def} \frac{1}{\sqrt[k]{N}}{X}_N\left(t/N^{1{-}1/k}\right)
  = Y_N(0){+}M_N(t){+}\lambda k t{-}k\int_0^t \frac{X_N(s)^{(k)}}{N}\diff s,
  \end{equation}
  where $(M_N(t))$ is a martingale whose  previsible increasing process is given by
  \[
  \croc{M_N}(t)=\frac{\lambda k^2 t}{N^{1{+}1/k}}{+}\frac{\mu k^2}{N^{1/k}}\int_0^t \frac{X_N(s)^{(k)}}{N}\diff s,
  \]
  therefore, with Relation~\eqref{Eeq} we get
  \[
  \E\left(\croc{M_N}(t)\right)\le \frac{\lambda k^2 t}{N^{1{+}1/k}}{+}\frac{\mu k}{N^{1/k}}\left(Y_N(0){+}\lambda k t\right).
  \]
  Doob's Inequality gives that the sequence of martingales  $(M_N(t))$ is converging in distribution to~$0$. By using again  Relation~\eqref{Eeq}, we get that,  for any $T{>}0$ and $\eps{>}0$, there exists $K$ such that
  \[
  \P\left(\sup_{t\le T}Y_N(t)\geq K\right)\le \eps.
  \]
 We can then use the criterion of the modulus of continuity, see Theorem~7.3 of~\citet{Billingsley}, to show that the sequence $(Y_N(t))$ is tight for the convergence in distribution. It is then easy to conclude the proof of the proposition. 
\end{proof}
When $k{=}1$, this is the classical result for the scaled $M/M/\infty$ queue that, for the convergence in distribution
\[
\lim_{N\rightarrow +\infty} \left( \frac{{X}_N(t)}{N}\right) = \left(\frac{\lambda}{\mu} + \left(\alpha{-}\frac{\lambda}{\mu}\right)e^{-\mu t}\right). 
\]
See Theorem 6.13 in \citet{Robert}.

The case $k{\ge}2$ is in fact more interesting, and more important for our study.  With Definition~\eqref{ScaledProcess}, the above proposition gives the asymptotic behavior of the process $(\overline{X}^N(t/N^{1-1/k}))$, i.e. on a slower timescale than the timescale  $(t)$ of interest in our paper. It is quite clear that $(\overline{X}_N(t))$ should be  close to the equilibrium of the function $(x(t))$, i.e.  close to $\ell_\infty{=}\sqrt[k]{\lambda/\mu}$. 

For such a process on a fast timescale, a convergence result of  $(\overline{X}^N(t))$ to $(\ell_\infty)$ is  classically formulated in terms of the convergence in distribution of its \emph{occupation measure}. See Section~\ref{Not}. Here, however,  a stronger result of convergence is a key ingredient in the proofs of tightness  for the convergence results of this paper. 
\begin{proposition}\label{PropConvEhr}
If $k{\ge}2$ and  $X_N(0){=}O(\sqrt[k]{N})$, then for any $0{<}\eta{<}T$,  and $\eps{>}0$, 
	\[
		\lim_{N\rightarrow +\infty} \P\left(\sup_{\eta\leq t{\leq}T} \left|\frac{X^N(t)}{\sqrt[k]{N}}{-}\ell_\infty \right|>\eps \right)=0
	        \]
holds  with $\ell_\infty{\steq{def}}\sqrt[k]{\lambda/\mu}$. 
\end{proposition}

\begin{proof}
The proof is carried out in two (similar)  steps:  with a stochastic upper bound of $\overline{X}^N(t){-}\ell_\infty$, and then, with a  stochastic lower bound of $\ell_\infty{-}\overline{X}^N(t)$. 

First, we show that the process reaches the neighborhood of $\ell_\infty$ before time $\eta{>}0$ with high probability. 
Let $\ell_1{>}\ell_\infty$, define
\[
S_N\steq{def}\inf\left\{t{\geq} 0: X_N(t)^{(k)}\leq (\ell_1)^kN\right\},
\] 
The integration of Relation~\eqref{Eeqd} gives 
\begin{align*}
	\E\left( X^N(\eta{\wedge}S_N)\right) &=x_N +k\E\left(\int_0^{\eta\wedge S_N} (\lambda N-\mu (X^N(u))^{(k)}) \diff u\right)\\
		&\leq C_0\sqrt[k]{N} +k\mu \left((\ell_\infty)^k{-}(\ell_1)^k\right)N\E(\eta{\wedge}\tau_N),
\end{align*}
for some constant $C_0$. Therefore we have for $N$ large enough, 
\[
	\E\left(\eta{\wedge}S_N\right)\leq \frac{C_0}{k\mu((\ell_1)^k{-}(\ell_\infty)^k)}N^{1/k-1}, 
\]
and therefore that $(\P(S_N{>}\eta))$ converges to $0$. 

With the strong Markov property of $(X_N(t))$, we can therefore assume that
\[
X_N(0)\leq y_N\steq{def}\ell_1 \sqrt[k]{N}{+}k{-}1.
\]
Let $(Z(t))$ be a birth and death process on $\N$ starting at $0$, with the transitions 
\[	
	x\to x{+}
	\begin{cases}
		{+}1 \quad \lambda,\\
		{-}1 \quad \mu (\ell_1)^k\text{ if } x{\geq} 1.
	\end{cases}
        \]
The process $(Z(t))$ is the process of the number of jobs of an $M/M/1$ queue with input rate $\lambda$ and service rate $\mu(\ell_1)^k$. See Chapter~5 of~\citet{Robert}. 
Since $\mu(\ell_1)^k{>}\lambda$, this process is positive recurrent.

We now construct a coupling of $(X_N(t))$ and $(Z(t))$ such that the relation 
\begin{equation}\label{couplingRelXZ}
	X_N(t)\leq y_N{+}kZ(Nt), \quad \forall t{\ge}0, 
\end{equation}
holds,  where  $(X_N(t))$ is  the solution of the SDE~\eqref{Eeqd} with initial point $x_N{\le}y_N$ and $(Z(t))$ is the solution of the SDE 
\[
	\diff Z(t)= \cal{P}_{01}\left((0, \lambda N), \frac{\diff t}{N}\right) -\ind{Z(t{-})>0}\cal{P}_{10}\left((0, \mu (\ell_1)^k N), \frac{\diff t}{N}\right)
        \]
with initial point at $0$. 

It is enough to prove Relation~\eqref{couplingRelXZ} by induction on the instants of jumps of the process $(X_N(t),Z(Nt))$ in the following way:  if the inequality holds at time $t_0$, then it also holds  at the instant of the next jump of the process $(X_N(t),Z(Nt))$ after time $t_0$.

Without loss of generality, we can assume that $t_0{=}0$ and $X_N(0){\le}y_N{+}kZ(0)$,  $t_1$ is the first instant of jump of $(X_N(t),Z(Nt))$.   Since both processes $(X^N(t))$ and $(kZ(Nt))$ have the same positive jump sizes at the same instants, we have only to consider jumps with negative sizes.

\begin{enumerate}
 \item If $X_N(0){\ge} y_N$, then $X_N(0)^{(k)}{\ge}(\ell_1)^k$. If at time $t_1$, there is a jump for $(Z(Nt))$ whose  size  is ${-}k$, it is due to the Poisson process $\cal{P}_{10}$. In view of the SDE for $(X_N(t))$, this implies that  there is also a jump ${-}k$ for $(X_N(t))$ at time $t_1$. Relation~\eqref{couplingRelXZ} will then also hold at the first instant of jump of $(X_N(t),Z(Nt))$. 

\item If $X_N(0){<} y_N$, if there is a negative jump of $(Z_N(Nt))$ at time $t_1$, Relation~\eqref{couplingRelXZ} will obviously hold at that instant. 
\end{enumerate}
All the other possibilities preserve clearly the desired inequality. 

Now, for $\ell_2$ such that  $\ell_2{>}\ell_1$, 
\[
\P\left(\sup_{0\leq t\leq T}\frac{X^N(t)}{\sqrt[k]{N}}\ge \ell_2  \right)\leq \P\left(\sup_{0\leq t\leq T} Z(Nt) \ge \frac{(\ell_2{-}\ell_1)\sqrt[k]{N}}{k} {-}1\right).
\]
If, for $0{<}\eps{<}\ell_2{-}\ell_1$,
\[
\tau_N\steq{def} \inf\{t\geq 0: Z(t)\geq \eps \sqrt[k]{N}\},
\]
with the last inequality,  we have therefore, for $N$ sufficiently large, 
\begin{equation}\label{Ehreq}
\P\left(\sup_{0\leq t\leq T}\frac{X^N(t)}{\sqrt[k]{N}}\ge \ell_2  \right)\leq \P(\tau_N\leq N T).
\end{equation}
Proposition 5.11 of~\citet{Robert} on the hitting times of a positive recurrent $M/M/1$ queue  gives that there exists $\rho{\in}(0,1)$ such that 
the sequence $(\rho^{\sqrt[k]{N}} \tau_N)$ converges in distribution to an exponentially distributed random variable. In particular 
\[
\limsup_{N\to+\infty}\P(\tau_N \leq  T N)=\limsup_{N\to+\infty}\P\left(\rho^{\sqrt[k]{N}}\tau_N\leq N\rho^{\sqrt[k]{N}}T\right)=0.
\]
Since $\ell_2$ is arbitrarily close to $\ell_\infty$, Relation~\ref{Ehreq} gives the relation for the upper bound. The other case uses the same ingredients. The proposition is proved. 
\end{proof}
With the same type of arguments, we can obtain the following corollary. 
\begin{corollary}\label{PropBoundEhr}
  For $k{\ge}2$, if the initial condition of  $(X_N(t))$ satisfies the relation
        \[
	\lim_{N\rightarrow +\infty} \frac{X_N(0)}{\sqrt[k]{N}} = \alpha > 0,
        \]
 and,  for $M{>}m{>}0$ such that $\alpha^k$, $\lambda/\mu{\in}(m,M)$.
	then
	\[
		\lim_{N\rightarrow +\infty}\P\left( \frac{X^N(s)^{(k)}}{N}\in(m, M), \forall s{\in}[0,T]\right)=1. 
	\]	
\end{corollary}
\section{Uniform Estimates}\label{Bounds}
This section is devoted to the proof of the fact that for any $T{>}0$, with high probability,  the scaled process $(\overline{X}_N(t))$ of Relation~\eqref{ScaledProcess} lives in a bounded domain of the interior of $\R_+^n$ uniformly on the time interval  $[0,T]$. Recall that since the components with index $i$ such that $k_i{\ge}2$ are on ``fast'' timescales, see Section~\ref{HiTS}, uniform estimates on a time interval are more challenging to establish.

\begin{theorem}\label{UnTh}
If $(X_N(t))$ is the Markov process associated to  the $k$-unary CRN of Definition~\ref{kDef}  whose matrix $R_{\kappa}$ is irreducible and with initial conditions satisfying Relation~\eqref{CondIni}, then for any $T{>}0$,  there exist two positive vectors $(m_i)$ and $(M_i)$ of $\R_+^n$ such that
\[
\lim_{N\to+\infty}\P\left({\cal E}_N\right)= \lim_{N\to+\infty}\P\left(\frac{X_i^N(t)^{(k_i)}}{N}\in(m_i,M_i), \forall i{\in}\{1,\ldots,n\}, \forall t{\le}T\right)=1,
\]
with, for $J{\subset}I$,
\begin{equation}\label{EN}
  \begin{cases}
    {\cal K}_J=&\displaystyle \left\{x{=}(x_i){\in}\left(\R_+\right)^{I^*}: \sqrt[k_i]{m_i}< (x_i)^{k_i}< \sqrt[k_i]{M_i}, \forall i{\in}J^*\right\}\\
    {\cal E}_N\steq{def}&\displaystyle \left\{ \overline{X}_N(t)\in {\cal K}_I, \forall t{\le}T\right\}.
  \end{cases}
\end{equation}
\end{theorem}
The important implication of this result is that, on the set ${\cal E}_N$,  \emph{every reaction} has a rate of the order of $N$. Note that because of the factorial term in the expression of the rate of the reactions, the event ${\cal E}_N$ is not the event 
\[
\left\{\frac{X_i^N(t)^{(k_i)}}{N}\in(m_i,M_i), \forall i{\in}\{1,\ldots,n\}, \forall t{\le}T\right\},
\]
however, when $N$ goes to infinity, both events have the same probability. 

 The proof of the theorem is done by considering the stopping time $H_N{\wedge}T_N$, where
\begin{equation}\label{HNTN}
        \begin{cases}
\displaystyle                H_N\steq{def} \inf\left\{ t\geq 0: \min_{i\in I^*}\frac{(X^N_i(t))^{(k_i)}}{m_i N} \leq 1 \right\},\\
\displaystyle                T_N\steq{def} \inf\left\{ t\geq 0:  \max_{i\in I^*}\frac{(X^N_i(t))^{(k_i)}}{M_i N} \geq 1\right\},
        \end{cases}     
\end{equation}
and prove that for any $T{>}0$, the sequence $(\P(H_N{\wedge}T_N{\le}T))$ converges to $0$.

The proof is done in several steps.  Results on convenient vectors $(m_i)$ and $(M_i)$ are established in Section~\ref{LAsec}. 
Proposition~\ref{BEprop} of Section~\ref{BFsec} proves the result when $\I{1}^*{=}\emptyset$, i.e. when $k_i{\ge}2$ for all $i{\in}I^*$.  Section~\ref{BGsec} concludes with the general case. A coupling argument with a set of  independent $M/M/\infty$ queues is used and  then Corollary~\ref{PropBoundEhr} of Section~\ref{Ehrenfest}. 

\subsection{Some Linear Algebra}\label{LAsec}
The notations and assumptions of  Section~\ref{Not} are used. 
\begin{proposition}\label{InvK}
If  $\kappa{\in}\Omega(I)$, then there exists a unique solution $\ell_\kappa{=}(\ell_{\kappa,i}){\in}(\R_+^*)^{I^*}$,
  such that, for $i{\in}I^*$,
\begin{equation}\label{system2}
\kappa_i^+(\ell_{\kappa,i})^{k_i}= \kappa_{0i}{+}\sum_{j\in I^*\setminus\{i\}}(\ell_{\kappa,j})^{k_j}\kappa_{ji},
  \end{equation}
furthermore,
\begin{equation}\label{S22}
 ((\ell_{\kappa,i})^{k_i})= M^R_\kappa \cdot\left(\frac{\kappa_{0i}}{\kappa_i^+}\right),
\end{equation}
  where $M^R_\kappa$ is an $I^*{\times}I^*$ matrix whose coefficients are non-negative and depend only on $\kappa_{ij}$, $i{\in}I^*$, $j{\in}I$. 
\end{proposition}
Recall that, from Relation~\eqref{KappaBullet}, if $i{\in}I^*$,
\[
\kappa_i^+=\kappa_{i0}{+}\sum_{j\ne i}\kappa_{ij}.
\]
\begin{proof}
The system~\eqref{system2} can be written as
  \[
  z\cdot R_{\kappa}=0,
  \]
  with  $z_0{=}1$ and $z_i{=}(\ell_{\kappa,i})^{k_i}$, for $i{\in}I^*$. This is simply the system of invariant measure equations for the Markov process associated to $R_\kappa$, introduced in Section \ref{kDef}. The irreducibility property gives the existence and uniqueness of such a solution $z$.

Relation~\eqref{S22} is just a linear algebra representation of this solution, based on the fact that the spectral radius of the matrix $R^*_\kappa{=}(\kappa_{ji}/\kappa_i^+, i, j{\in}I^*)$ is strictly less than $1$, which is a consequence of the irreducibility of $R_{\kappa}$,
  \[
\left(\kappa_i^+(\ell_{\kappa,i})^{k_i}\right) = \left(\sum_{m=0}^{+\infty} (R^*_\kappa)^m\right)\cdot (\kappa_{0i}).
  \]
\end{proof}
The following proposition is  a key result used in a coupling in the proof of Theorem~\ref{UnTh}.
\begin{proposition}\label{LemmaMiExist}
If $\kappa{\in}\Omega(I)$,  then for any $(\alpha_i){\in}(\R_+^*)^{n}$,  there exist two vectors $(m_i)$ and  $(M_i){\in}\R_+^n$ such that,
\begin{equation}\label{maM}
0<m_i<\alpha_i^{k_i}<M_i,\quad \forall i\in I^*,
\end{equation}
and   
  \begin{align}
		M_i\kappa_{i}^+  &> \kappa_{0i} +\sum_{j\in I^*\setminus\{i\}}M_j\kappa_{ji}, \label{EqM}\\
		m_i\kappa_{i}^+  &<\kappa_{0i}+ \sum_{j\in I^*, \ d(j)<d(i)}m_j\kappa_{ji}, \label{Eqm}
  \end{align}
where $d$ is the distance of Definition \ref{definitiond}. 
\end{proposition}
It should be noted that Relations~\eqref{EqM} and~\eqref{Eqm} are not symmetrical, because of the restriction on the summation using the distance $d$. The result will used for the vector $(\alpha_i)$ associated to the initial conditions, see Relation~\eqref{CondIni}.
\begin{proof}
Let $(z_i){=}((\ell_{\kappa,i})^{k_i})$ be the solution of the system of the type~\eqref{system2}, for $i{=}1$,{\ldots}, $n$, 
          \[
                 z_i\kappa_{i}^+ = 1+ \sum_{j\in I^*\setminus\{i\}}z_j\kappa_{ji}. 
                 \]
Relations $\alpha_i^{k_i}{<}M_i$ and~\eqref{EqM} hold if we take $M_i{=}\rho z_j$, with
                 \[
                 \rho >  \max\left(\frac{\alpha_i^{k_i}}{z_i}, \kappa_{0i}: i{=}1, \ldots, n,\right).
                 \]
                 The construction of  $(m_i)$ for the lower bounds  is done by induction on the values of $d(i)$.

                 If $i{\in}I^*$ is such that $d(i){=}1$, then necessarily $\kappa_{0i}{>}0$, then we can take $m_i$ so that
                 \[
                 0<m_i<\min\left(\frac{\kappa_{0i}}{\kappa_i^+},\alpha_i^{k_i}\right).
                 \]
                 If $d(i){=}p{\ge}2$, then there exists $j{\in}I^*$ such that $d(j){=}p{-}1$ and $\kappa_{ji}{>}0$, therefore we can take $m_i{>}0$ such that 
                 \[
m_i< \min\left( \alpha_i^{k_i},\frac{1}{\kappa_i^+}\left(\kappa_{0i}+ \sum_{j\in I^*, d(j)<d(i)}m_j\kappa_{ji}\right)\right), 
\]
since the sum of the second term is strictly positive. 
The proposition is proved. 
\end{proof}

\subsection{CRN with Only Fast Processes}\label{BFsec}
In this section it is assumed that $\I{1}^*$ is empty, i.e. $k_i{\ge}2$ for all $i{\in}\{1,\ldots,n\}$.
\begin{proposition}\label{BEprop}
  If $\I{1}^*{=}\emptyset$, then there exist two vectors $(m_i)$ and $(M_i)$ with positive coordinates such that for any $T{>}0$, the sequence  $(\P\left( {\cal E}_N\right))$ is converging to $1$, 
where ${\cal E}_N$ is the event defined by Relation~\eqref{EN}.
\end{proposition}
\begin{proof}
  Let  $H_N$ and $T_N$ be  the stopping times defined by Relation~\eqref{HNTN}. 
We start with the stopping time $T_N$. We take $(M_i)$ of Proposition~\ref{LemmaMiExist} satisfying Relations~\eqref{maM} and~\eqref{EqM}. Let $(Y_N(t))=(Y_i^N(t))$ be the solution of the SDE
\begin{multline*}
	\diff Y_i^N(t)= k_i\cal{P}_{0i}\left( (0, \kappa_{0i}N), \diff t\right)+ \sum_{j\in I^*\setminus\{i\}} k_i \cal{P}_{ji}\left(\left( 0, \kappa_{ji}M_j N\right), \diff t\right)\\
	-\sum_{j\in I\setminus\{i\}} k_i \cal{P}_{ij}\left(\left( 0, \kappa_{ij}(Y_i^N(t{-}))^{(k_i)}\right), \diff t\right),
\end{multline*}
with $Y_N(0){=}X_N(0)$. Note that we have necessarily that $Y_i^N(t){-}X_i^N(t){\in}k_i\Z$, for all $i{\in}I^*$ and $t{\ge}0$.

We  prove  that, for all $t{<}T_N$, the relations $X_i^N(t){\le}Y_i^N(t)$ hold for any $i{\in}I^*$. This is done by induction on the sequence of the instants of jumps of the process $((X_i^N(t),Y_i^N(t)),i{\in}I^*)$ in the time interval $[0,T_N]$.  As  in the proof of Proposition~\ref{PropConvEhr}, we assume that $X_i^N(0){\le}Y_i^N(0)$ and $X_i^N(0)^{(k_i)}{\leq} M_iN$ hold for all $i{\in}I^*$,  and denote by $t_1$  the instant of the first jump of the process $(X_i^N(t),Y_i^N(t),i{\in}I^*)$. We show that the above inequalities also hold at time $t_1$.

For all $i{\in}I^*$,   we have $(X^N_i(0))^{(k_i)}\leq M_iN$,  hence, for, $j{\in}I^*$ and $t{\ge}0$, 
\[
\cal{P}_{ji}\left(\left( 0, \kappa_{ji}(X^N_i(t))^{(k_i)}\right){\times}[0, t]\right) \leq \cal{P}_{ji}\left(\left( 0, \kappa_{ji}M_j N\right){\times}[0, t]\right),
\] 
and if $(X_i^N(t))$ has a jump up at $t_1$ due to ${\cal P}_{ji}$, so does $(Y_i^N(t))$. Consequently, the inequality is clearly preserved if the size of the first jump is positive. 

If $t_1$ is an instant of a jump with negative size for $(Y_i^N(t))$, if $X_i^N(0){<}Y_i^N(0)$, then necessarily $Y_i^N(0){-}X_i^N(0){\ge}k_i$, the relation $X_i^N(t_1){\le}Y_i^N(t_1)$ is therefore satisfied. All the other possibilities for $t_1$ clearly preserve the desired relations. Our assertion has been established. 

 For $i{\in}I^*$, the process $(Y_i^N(t))$ has the same distribution as the process of a generalized $M/M/\infty$  queue, introduced in Section \ref{Ehrenfest}, with arrival rate $\lambda_i N$ and departure rate $\mu_i$ given by
 \[
 \lambda_i=\kappa_{0i}{+}\sum_{j{\in}I^*{\setminus}\{i\}} M_j\kappa_{ji}, \quad \mu_j=\kappa_i^+.
 \]
	We have $M_i{>} \lambda_i/\mu_i$ for all $i{\in}I^*$ because of Relation~\eqref{EqM}. Since all $k_i$'s are greater than $2$, Corollary~\ref{PropBoundEhr} applied to these $n$ generalized $M/M/\infty$ queues  shows  that the relation
 \[
 \lim_{N\to+\infty}\P(T_N\le T)=0
 \]
holds. We now take care of the stopping time $H_N$.  A vector $(m_i)$ satisfying Relations~\eqref{maM} and~\eqref{Eqm} of Proposition~\ref{LemmaMiExist}  is fixed.
 Let $(Z_N(t)){=}(Z_i^N(t))$ be the solution of the SDE
\begin{multline*}
	\diff Z_i^N(t)= k_i\cal{P}_{0i}\left( (0, \kappa_{0i}N), \diff t\right)+ \sum_{\substack{j\in I^*\\ d(j)< d(i)}} k_i \cal{P}_{ji}\left(\left( 0, \kappa_{ji}m_j N\right), \diff t\right)\\
	-\sum_{j\in I\setminus\{i\}} k_i \cal{P}_{ij}\left(\left( 0, \kappa_{ij}(Y_i^N(t{-}))^{(k_i)}\right), \diff t\right),
\end{multline*}
with $Z_N(0){=}X_N(0)$.
It is easily seen by induction on the sequence of the instants of jumps of the process $(X_i^N(t),Z_i^N(t))$  that the  relation $X_i^N(t){\ge}Z_i^N(t)$ holds for all $t{<}H_N$ and $i{\in}I^*$.

 For $i{\in}I^*$, the process $(Z_i^N(t))$ has the same distribution as the process of a generalized $M/M/\infty$  queue  with arrival rate $\lambda_i N$ and departure rate $\mu_i$ given by
 \[
 \lambda_i\steq{def} \kappa_{0i}{+}\sum_{\substack{j\in I^*\\ d(j)< d(i)}} m_j\kappa_{ji}, \quad \mu_j\steq{def}\kappa_i^+.
 \]
Since the vector $(m_i)$ has been chosen so that $m_i{<}\lambda_i/\mu_i$ holds for all $i{\in}I^*$, we can conclude in the same way as before using Corollary~\ref{PropBoundEhr}.
The proposition is proved. 
\end{proof}
\subsection{Proof of Theorem~\ref{UnTh}}\label{BGsec}

We first take care of the indices in the set $\I{1}^*$. We define 
    \begin{equation}\label{mM1}
      \begin{cases}
      \displaystyle m_1^1=\frac{1}{2}\alpha_{\rm min}\exp({-}\kappa_{\rm max}^+ T),\\
      \displaystyle      M_1^1=2k_{\rm max}\left( \kappa_0^+ T{+}\sum_{j\in \I{1}^*}\alpha_j\right),
      \end{cases}     
    \end{equation}
    with $x_{\rm max/min}{=}\max{/}\min(x_i,1{\le}i{\le}n)$ for $x{\in}\R_+^n$. 

We show here that for all $i{\in}\I{1}^*$, we can choose $m_i{=}m_1^1$ and $M_i{=}M_1^1$. For all $i{\in}I^*$, it is easily seen that the following upper bound, for $t{\ge}0$, 
\begin{equation}\label{UpI1}
\sup_{t\le T}\sum_{i{\in}\I{1}^*}k_i X_i^N(t)\leq
k_{\rm max}\left(\sum_{i{\in}\I{1}^*} x_i^N+\sum_{i{\in}\I{2+}^*} x_i^N+\sum_{i{\in}\I{1}^*}{\cal P}_{0i}([0,\kappa_{0i}N]{\times}[0,T])\right).
\end{equation}
holds. The right-hand side of the last relation divided by $N$ converges almost surely to
\[
k_{\rm max}\left(\kappa_{0}^+T{+}\sum_{i{\in}\I{1}^*} \alpha_i\right),
\]
hence
\begin{equation}\label{1up}
\lim_{N\to+\infty}\P\left(\sup_{t\leq T} \max_{i{\in}\I{1}^*}\frac{X_i^N(t)}{N}\ge M_1^1 \right)=0.
\end{equation}

Since the lifetime of a molecule of type $i{\in}\I{1}^*$ is  exponentially distributed   with parameter $\kappa_i^+$,  the number of  species  $i$ at time $T$ is stochastically greater than
\[
	\sum_{k=1}^{x_i^N} \ind{E_k^{i+}\ge T},
\]
where $(E_k^{i+})$ is a sequence of i.i.d. exponential random variables with parameter $\kappa_i^+$. This last quantity divided by $N$ converges almost surely to $\alpha_i\exp(-\kappa_i^+T)$. We therefore obtain the relation
\begin{equation}\label{1dow}
\lim_{N\to+\infty}\P\left(\inf_{t\leq T} \min_{i{\in}\I{1}^*}\frac{X_i^N(t)}{N}\le  m_1^1 \right)=0.
\end{equation}

From Relations~\eqref{1up} and~\eqref{1dow}, for  node $i{\in}\I{2+}^*$, the input rate from node  $j{\in}\I{1}^+$ on the time interval $[0,T]$ is, with high probability,  upper bounded by $\kappa_{ji}M_1^1$ and lower bounded by $\kappa_{ji}m_1^1$.

Define $\overline{\kappa}{=}(\overline{\kappa}_{ij},i,j{\in}\I{2+})$ and 
 $\underline{\kappa}{=}(\underline{\kappa}_{ij},i,j{\in}\I{2+})$, by, for $i{\in}\I{2+}$, 
 \[
					\begin{cases}
									\overline{\kappa}_{ij}= \underline{\kappa}_{ij}= \kappa_{ij}, \quad j{\in}\I{2+};\\
									\overline{\kappa}_{i0}=\underline{\kappa}_{i0}= \kappa_{i0} +\sum_{j\in \I{1}^*}\kappa_{ij};\\
									\overline{\kappa}_{0i}= \kappa_{0i} +\sum_{j\in \I{1}^*}\kappa_{ji}M_1^1; \\
									\underline{\kappa}_{0i}= \kappa_{0i} +\sum_{j\in \I{1}^*}\kappa_{ji}m_1^1. \\
					\end{cases}
	\]
        Using a coupling argument, one can define the Markov processes $(Y^{2}_N(t))$, respectively  $(Z^{2}_N(t))$, associated  to the $k$-Unary CRN with species $\I{2,+}^*$, with complexes $(k_iS_i, i\in \I{2+}^*)$ and constant of reactions $\overline{\kappa}$, respectively $\underline{\kappa}$, both starting at $X^N_{[2, +]}(0)$ and that verify for all $t\leq T_N\wedge H_N$, 
					\[
							Z_i^{N,2}(t)\leq X_i^N(t)\leq Y^{N, 2}_i(t), \quad \forall i\in \I{2+}^*. 
					\]
				
Since $\overline{\kappa}{\in}\Omega(\I{2+})$,  Proposition~\ref{BEprop} applied to the process $(Y^2_N(t))$ shows that there exists a vector $(M_i,{\in}\I{2+})$, such that
        \[
        \lim_{N\to+\infty}\P\left( \overline{X}_N(t)\in\prod_{i=1}^N \left(0,\sqrt[k_i]{M_i}\right), \forall t{\le}T\right)=1.
        \]
        Similarly, by considering $\underline{\kappa}$, there exists a vector $(m_i,{\in}\I{2+})$ with positive components such that
        \[
        \lim_{N\to+\infty}\P\left( \overline{X}_N(t)\in\prod_{i=1}^N \left(\sqrt[k_i]{m_i},\sqrt[k_i]{M_i}\right), \forall t{\le}T\right)=1.
        \]
The theorem is proved.         
\section{CRN with only fast processes}\label{FastSec}
When $\I{1}^*$ is empty, i.e. $k_i{\ge}2$ for all $i{\in}\{1,\ldots,n\}$, the time evolutions of all species are fast processes, see Section~\ref{HiTS}. Theorem~\ref{Theorem} is only about the convergence in distribution of the sequence of occupation measures $(\Lambda_N)$ on $\R_+{\times}\R_+^n$ defined by Relation~\eqref{OccMeas}. The absence of chemical species $i$ such that $k_i{=}1$ gives a kind of instantaneous equilibrium property in the sense that the limit in distribution of $(\Lambda_N)$ is homogeneous with respect to the first coordinate, the time coordinate. The main result of this section is Theorem~\ref{Thk2} which is simply Theorem~\ref{Theorem} stated in this context. The motivation of such a separate proof  is that it is focused, in our view, on the key argument of the general proof.   The identification of possible limits of $(\Lambda_N)$  is done by induction via the use of an entropy function. The proof of the general case follows also such line  but in a ``non-homogeneous'', technically more complicated, context. 

\subsection{Tightness of $(\Lambda_N)$}
We first establish the  tightness of  $(\Lambda_N)$ for the convergence in distribution in the general case.
\begin{proposition}\label{PropTightkgeq2}
If the subset  $\I{1}^*$ is empty and if the initial conditions satisfy Relation~\eqref{CondIni}, then the sequence of measure valued processes $(\Lambda_N)$ on $[0,T]\times (\R_+^*)^{I^*}$ is tight for the convergence in distribution.  Any limiting point $\Lambda_\infty$ can be expressed as, 
	\begin{equation}\label{ConvKurtzkgeq2}
		\croc{\Lambda_\infty, f}=\int_{[0, T]\times {\cal K}_I}f(s, x) \pi_s(\diff x) \diff s, 
	\end{equation}	
for any function $f{\in}{\cal C}_c([0, T]{\times}(\R_+)^{I^*})$, where $(\pi_s)$ is an optional process with values in  ${\cal P}({\cal K}_I)$, the set of probability measures on the compact subset ${\cal K}_I$ defined by Relation~\eqref{EN}. 
\end{proposition}
See~\citet{Dawson} for a presentation of the convergence in distribution of measure-valued processes. The optional property of $(\pi_s)$ is used only to have convenient measurability properties  so that time-integrals with respect to $(\pi_s,s{>}0)$ are  indeed  random variables. See Section VI.4 of~\citet{Rogers2}.
\begin{proof}
  We take the vectors $(m_i)$ and $(M_i)$ of  Theorem~\ref{UnTh}, and ${\cal K}_I$ the compact set of $\R_+^n$ and ${\cal E}_N$  the event  defined in Relation~\eqref{EN}.
Since $\Lambda_N([0,T]{\times}{\cal K}_I){\ge}T {\mathbbm 1}_{{\cal E}_N}$, with Theorem~\ref{UnTh}, we obtain the relation
  \[
  \lim_{N\to+\infty}\E\left(\Lambda_N\left([0,T]{\times}{\cal K}_I\right)\right)=T.
  \]
Lemma~1.3 of~\citet{Kurtz1992} gives that the sequence of random measures $(\Lambda_N)$ is tight for the convergence in distribution, and Lemma~1.4 of the same reference gives the representation~\eqref{ConvKurtzkgeq2}. The proposition is proved. 
\end{proof}
In the following we assume that $\Lambda_\infty$ is a limit of a subsequence $(\Lambda_{N_r})$ with the representation~\eqref{ConvKurtzkgeq2}.
\begin{lemma}\label{lemaux}
  If $f$ is a continuous function on $\R_+^{I^*}$, then the relation
  \[
\lim_{r\to+\infty} \left(\int_0^t f\left(\overline{X}_{N_r}(s)\right)\diff s\right)
  =\left(\int_0^t \int_{\R_+^{I^*}}f(x) \pi_s(\diff x) \diff s, \right)
  \]
  holds for the convergence in distribution of processes. 
\end{lemma}
\begin{proof}
  This is a straightforward use of  the criterion of modulus of continuity,  see Theorem~7.3 of~\citet{Billingsley}, and of Theorem~\ref{UnTh}. 
For $s{\le}t$, on the event ${\cal E}_N$, we have
\[
\int_s^t f\left(\overline{X}_{N_r}(s)\right)\diff s\le 2(t{-}s)\sup_{x{\in}{\cal K}_I}|f(x)|,
\]
with the notations of  Relation~\eqref{EN}. We conclude with the identification of the finite marginals. 
\end{proof}

As we have seen in Section~\ref{HiTS}, for $i{\in}I^*$, the value of $k_i$ gives in fact the natural timescale of the process $(\overline{X}_i^N(t))$.  On the event ${\cal E}_N$,  see Relation~\eqref{EN},  every reaction has a rate of order $N$, in particular, the rate at which  the process $(X_i^N(t))$ jumps of $\pm k_i$ is of order $N$. With the scaling in space of the process,  $(\overline{X}^N_i(t))$ is significantly changed  when there are $N^{1/k_i}$ reactions changing $(X_i^N(t))$, and therefore after a duration of time of the order of $N^{1/k_i-1}$.  If  for two species $i$ and $j$, $k_i{>}k_j$, then the process $(\overline{X}_i^N(t))$ changes more rapidly  than the process $(\overline{X}_j^N(t))$.  
   
From now on in this section it is assumed that $\I{1}^*$ is empty. 
\subsection{A Limiting Equation}
For a function $f\in \cal{C}_c^2\left((\R_+)^{I^*}\right)$, the SDE~\eqref{SDE} gives directly, for $t{\in}[0, T]$,  
\begin{multline}\label{fSDE}
	f\left(\overline{X}_N(t)\right)= f\left(\overline{X}_N(0)\right) +M_{f,N}(t) +\int_0^{t}\sum_{i\in I^*} \kappa_{0i}N\nabla_{\frac{k_i}{N^{1/k_i}}e_i}(f) (\overline{X}_N(s)) \diff s\\
		+\int_0^{t}\sum_{\substack{i,j\in I,\\ i\neq 0}} \kappa_{ij}(X_i^N(s))^{(k_i)}\nabla_{-\frac{k_i}{N^{1/k_i}}e_i+\frac{k_j}{N^{1/k_j}}e_j}(f) (\overline{X}_N(s)) \diff s,
\end{multline}
with the notations
\begin{itemize}
	\item for $x$, $a{\in} \R^{I^*}$, $\nabla_{a}(f)(x){=}f(x{+}a){-}f(x)$;
	\item for $i{\in} I^*$, $e_i$ is the $i$-th unit vector of $\R^{I^*}$, and the convention $e_0{=}0$,
\end{itemize}
and $(M_{f,N}(t))$ is local martingale whose previsible increasing process is given by, for $t{\le}T$,
\begin{multline}\label{CfSDE}
\croc{M_{f,N}}(t)=\int_0^{t}\sum_{i\in I^*} \kappa_{0i}N\left(\nabla_{\frac{k_i}{N^{1/k_i}}e_i}(f) (\overline{X}_N(s))\right)^2 \diff s\\
		+\int_0^{t}\sum_{\substack{i,j\in I,\\ i\neq 0}} \kappa_{ij}(X_i^N(t))^{(k_i)}\left(\nabla_{-\frac{k_i}{N^{1/k_i}}e_i+\frac{k_j}{N^{1/k_j}}e_j}(f) (\overline{X}_N(s))\right)^2 \diff s
\end{multline}

\begin{proposition}\label{propEqLimit}
If the subset  $\I{1}^*$ is empty and $(\Lambda_{\infty})$ is a limiting point of  $(\Lambda_N)$ with the representation~\eqref{ConvKurtzkgeq2}, then, 
for any $p{\geq} 2$ and $f{\in}\cal{C}^2({\cal K}_{\I{2,p}})$,  almost surely, the relation
\begin{equation}\label{eqLimitMeas}
		\int_0^t \int_{{\cal K}_I}\sum_{i\in \I{p}^*} \left(\kappa_{0i}{+}\sum_{j\in I^*{\setminus}\{i\}}\kappa_{ji} x_j^{k_j} {-}\kappa_{i}^+ x_i^p  \right) \frac{\partial f}{\partial x_i}(x_{[2,p]}) \pi_s(\diff x)\diff s =0,
	\end{equation}
holds for all $t{\in}[0,T]$.
\end{proposition}
Recall the conventions  $x_{[2,p]}\hspace{0mm}{=}(x_i,i{\in}\I{2,p}^*)$ for $x{\in}(\R_+)^{I^*}$,  see Section~\ref{Not}.
\begin{proof}
 It is assumed that $\I{p}^*{\neq}\emptyset$. Let $f{\in}\cal{C}^2((\R_+)^{\I{2,p}^*})$. To simplify expressions in this proof,  we will make the slight abuse of notation, $f(x){=}f(x_{[2,p]})$ for $x{\in} (\R_+)^{I^*}$.
  
Since our goal is of characterizing the process $(\pi_t)$, by Theorem~\ref{UnTh}, without loss of generality, we can assume that the support of the function $f$ is included in ${\cal K}_I$ defined in Relation~\eqref{EN}. Similarly, from now on, all relations are considered on the event ${\cal E}_N$ whose probability is arbitrarily close to $1$ as $N$ gets large. 
In particular the process $(\overline{X}_N(t),t{\in}[0,T])$ has values in ${\cal K}_I$. 

For $t{\leq} T$, Relation~\eqref{fSDE} can be rewritten as, 
\begin{multline}\label{fSDE1}
  \frac{f\left(\overline{X}_N(t)\right)}{N^{1-1/p}}- \frac{f\left(\overline{X}_N(0)\right)}{N^{1-1/p}} - \frac{M_{f,N}(t)}{N^{1-1/p}}\\
  =\int_0^{t}\sum_{i\in \I{2,p}^*} \left(\kappa_{0i}{+}\sum_{j{\not\in}\I{2,p}}\kappa_{ji}\frac{(X_j^N(t))^{(k_j)}}{N}\right)N^{1/p}\nabla_{\frac{k_i}{N^{1/k_i}}e_i}(f) (\overline{X}_N(s)) \diff s\\
  +\int_0^{t}\sum_{i\in \I{2,p}^*}\left(\kappa_{i0}+\sum_{j{\not\in\I{2,p}}}\kappa_{ij}\right)\frac{(X_i^N(t))^{(k_i)}}{N}N^{1/p}\nabla_{-\frac{k_i}{N^{1/k_i}}e_i}(f) (\overline{X}_N(s)) \diff s\\
   +\int_0^{t}\sum_{i\in \I{2,p}^*}\sum_{j{\in\I{2,p}^*}{\setminus}\{i\}}\kappa_{ij}\frac{(X_i^N(t))^{(k_i)}}{N}N^{1/p}\nabla_{-\frac{k_i}{N^{1/k_i}}e_i+\frac{k_j}{N^{1/k_j}}e_j}(f) (\overline{X}_N(s)) \diff s.
\end{multline}
For $a$, $b{\ge}0$, there exist constants $C_0$ and $C_1$ such that
\begin{equation}\label{TechnicEqApprox}
\max_{i{\in}I^*}\sup_{x{\in}{\cal K}_I}\left|x^{k_i}{-}\frac{(\sqrt[k_i]{N}x)^{(k_i)}}{N}\right|\le \frac{C_0}{N^{1/k_i}},
\end{equation}
and, for any $i$, $j{\in}I^*$,
\begin{multline*}
\sup_{x{\in}{\cal K}_I}\left|\nabla_{-\frac{a}{N^{1/k_i}}e_i+\frac{b}{N^{1/k_j}}e_j}(f)(x){+}\frac{a}{N^{1/k_i}}\frac{\partial f}{\partial x_i}(x)
    {-}\frac{b}{N^{1/k_j}}\frac{\partial f}{\partial x_j}(x)\right|\\
    \le C_1\left(\frac{a}{N^{1/k_i}}{+}\frac{b}{N^{1/k_j}}\right).
\end{multline*}

We get that, for $i{\in}\I{2,p}$, the  processes
\[
\left(N^{1/p}\nabla_{\pm\frac{k_i}{N^{1/k_i}}e_i}(f) (\overline{X}_N(t)), t\le T \right)
\]
vanish if $k_i{\ne}p$. With   the definition~\eqref{HNTN},  Relation~\eqref{CfSDE} and Doob's Inequality give that the martingale
$(M_{f,N}(t{\wedge}T_N)/N^{1-1/p})$ converges in distribution to $0$ and so $(M_{f,N}(t)/N^{1-1/p})$ by Theorem~\ref{UnTh}.

Relation~\eqref{fSDE1} becomes 
\begin{multline*}
\int_0^{t}\sum_{i\in \I{p}^*} \left(\kappa_{0i}{+}\sum_{j{\in}I^*{\setminus}\{i\}}\kappa_{ji}\left(\overline{X}_j^N(t)\right)^{k_j}\right)p\frac{\partial f}{\partial x_i}(\overline{X}_N(s)) \diff s\\
  -\int_0^{t}\sum_{i\in \I{p}^*}\kappa_{i}^+\left(\overline{X}_i^N(t)\right)^{k_j}p\frac{\partial f}{\partial x_i}(\overline{X}_N(s)) \diff s=U_N(t),
\end{multline*}
where $(U_N(t))$ is a process converging in distribution to $0$.  This relation can be written in terms of occupation measure $\Lambda_N$, it is easy to conclude the proof of the proposition with the help of Lemma~\ref{lemaux}.
\end{proof}

\subsection{A Convex Function on ${\cal K}_I$}
\begin{definition}
If $\kappa{\in}\Omega(I)$,  the function  $F_\kappa$ is defined by, for $z{=}(z_i){\in}{\cal K}_I$,
	\begin{equation}\label{DefF}
	F_\kappa(z)\steq{def} \sum_{i\in I^*} \left(\kappa_{i}^+z_i{-}\kappa_{0i}{-} \sum_{j\in I^*{\setminus}\{i\}}\kappa_{ji}z_j\right) \ln \left(\frac{z_i}{(\ell_{\kappa,i})^{k_i}}\right),
	\end{equation}
where  ${\cal K}_I$ is defined by Relation~\eqref{EN} and $\ell_\kappa{=}(\ell_{\kappa,i}){\in}\R_+^{I^*}$ is the unique solution of the system~\eqref{system2} of Proposition~\ref{InvK}.
\end{definition}
\begin{proposition}\label{PropConvex}
The function $F_\kappa$ is non-negative, strictly convex on ${\cal K}_I$, with a unique  minimum $0$ at $z{=}((\ell_{\kappa,i})^k)$, furthermore
the mapping $(\kappa, z){\mapsto} F_\kappa(z)$ is continuous on $\Omega(I)\times{\cal K}_I$. 
\end{proposition}
\begin{proof}
The existence and uniqueness of $\ell_\kappa$, solution of a non-singular linear system, has been seen in Proposition~\ref{InvK}. The continuity of $\kappa{\mapsto}\ell_\kappa$ on $\Omega(I)$ gives the continuity of $(\kappa,z){\mapsto}F_{\kappa}(z)$.

We now calculate the Hessian matrix of $F_\kappa$. For $i{\in}I^*$, we have, for $z{\in}{\cal K}_I$, 
\begin{multline*}
\frac{\partial F_\kappa}{\partial z_i} (z)= \kappa_{i}^+\ln\left(\frac{z_i}{(\ell_{\kappa,i})^{k_i}}\right) +\frac{1}{z_i}\left(\kappa_{i}^+z_i- \kappa_{0i}{-}\sum_{m\in I^*\setminus\{i\}}\kappa_{mi}z_m\right)\\{-}\sum_{m\in I^*\setminus\{i\}}\kappa_{im} \ln\left(\frac{z_m}{(\ell_{\kappa,m})^{k_m}}\right).
\end{multline*}
Relation~\eqref{system2} gives that this quantity is indeed null at $z{=}((\ell_{\kappa,m})^{k_m})$. For $j{\in}I^*$, $j{\neq}i$, we have the relation 
	\[
			\frac{\partial^2 F_\kappa}{(\partial z_i)^2} (z)=\frac{1}{z_i^2} \left(\kappa_{i}^+ z_i {+}\kappa_{0i}{+}\hspace{-4mm}\sum_{m\in I^*\setminus\{i\}}z_m\kappa_{mi}\right), \quad 
		\frac{\partial^2 F_\kappa}{\partial z_i\partial z_j} (z)=-\frac{\kappa_{ij} z_i+\kappa_{ji}z_j}{z_iz_j}.
	\]
Let ${\cal H}_\kappa(z)$ be the Hessian matrix of $F_\kappa$ at $z{\in}{\cal K}_I$. 
For $u{=}(u_i){\in}\R^{I^*}$, with the notation $\gamma_{ij}{=}\kappa_{ij}z_i{+}\kappa_{ji}z_j$ for $i{\ne}j$, the associated quadratic form at $u$ is  given by
	\begin{align*}
		u^t F_\kappa(z)u&= -\sum_{i\in I^*} \sum_{j\in I^*\setminus \{i\}} \gamma_{ij} \frac{u_iu_j}{z_iz_j} +\sum_{i\in I^*} (\kappa_{i0}z_i{+}\kappa_{0i})\frac{u_i^2}{z_i^2}+\sum_{i\in I^*} \sum_{j\in I^*\setminus\{i\}} \gamma_{ij} \frac{u_i^2}{z_i^2}\\
		&=\sum_{i\in I^*} (\kappa_{i0}z_i{+}\kappa_{0i})\frac{u_i^2}{z_i^2}+\frac{1}{2}\sum_{i\in I^*} \sum_{j\in I^*\setminus\{i\}} \gamma_{ij} \left(\frac{u_i}{z_i}{-}\frac{u_j}{z_j}\right)^2.
	\end{align*}
This last expression is positive for any non-zero element $u{=}(u_i){\in}\R^{I^*}$. The function $F_\kappa$ is strictly convex. This concludes the proof of the proposition.
\end{proof}

\subsection{Identification of the Limit}
We can now state the main convergence result of this section.
\begin{theorem}\label{Thk2}
  If $\kappa{\in}\Omega(I)$ and the subset  $\I{1}^*$ is empty, if Relation~\eqref{definitiond} holds for the initial conditions,   then the sequence  $(\Lambda_N)$ is converging in distribution to $\Lambda_\infty$, such that, almost surely, for any function $f{\in}\cal{C}_c(\R_+{\times}(\R_+^*)^{I^*})$, the relation 
 \begin{equation}\label{Thk2equation}
		\int f(s, x)\Lambda_\infty(\diff s, \diff x)= \int_0^{+\infty}f(s, \ell_\kappa)\diff s, 
	\end{equation}
holds,  where  $\ell_{\kappa}{=}(\ell_{\kappa,i})$ is the unique solution of the system~\eqref{system2} of Proposition~\ref{InvK}.
\end{theorem}

The proof is carried out by induction on the ``speed'' of the different processes. We start by the identification of the fastest species, with the largest $k_i$, and identify step by step each set $\I{p}^*$. One of the difficulties is that we  have only the functional equation, Relation~\eqref{eqLimitMeas}, to identify all the species in the set $\I{p}^*$ for each $p\geq 2$. A convex function, related to a relative entropy functional,  will be used to identify them simultaneously.

\begin{proof}
 Let $m_0{\ge}1$ and $(p_a){\in}\N^m$ such that $2{\le}p_{m_0}{<}{\cdots}{<}p_2{<}p_1$ and 
  \[
  \{k_i,i{\in}I^*\}=\{p_a,a{=}1,\ldots, m_0\},
  \]
  in particular, we have 
  \[
  I^*=\bigcup_{a=1}^{m_0}\I{p_a}^* \text{ and } I{=}\I{2,p_{1}}. 
  \]
We will proceed by induction on $m_0$ to prove that a random measure  $\Lambda_\infty$ that verifies Relation~\eqref{eqLimitMeas} is expressed by Relation~\eqref{Thk2equation}.

We first consider the species of the set  $\I{p_1}^*$ associated to the fastest processes of $(X_N(t))$. 
With the notations of Relation~\eqref{EN}, Relation~\eqref{eqLimitMeas}  gives, for $T{>}0$ and $p_1$, the identity
\begin{equation}\label{weq1}
\int_0^{T} \int_{{\cal K}_I} \sum_{i\in \I{p_1}^*} K_i\left[x_{[2,p_2]}\right](x_{[p_1]}{}^{p_1}) \frac{\partial f}{\partial x_i}(x) \pi_s(\diff x)\diff s=0
\end{equation}
holds almost surely for $f{\in}\cal{C}^2({\cal K}_I)$, with, for $y{\in}{\cal K}_{\I{2,p_2}}$,  $z{\in}{\cal K}_{\I{p_1}}$ and $i{\in}\I{p_1}^*$, 
\[
K_i[y](z)\steq{def} \kappa_{0i}+\sum_{j\in \I{2,p_2}^*} y_j^{k_j}\kappa_{ji} +\sum_{j\in \I{p_1}^*\setminus\{i\}}z_j\kappa_{ji} -\kappa_{i}^+ z_i.
\]
 and the notation $z^{p_1}{=}(z_i{}^{p_1})$. 

For $y{\in}{\cal K}_{\I{2,p_2}}$, we introduce an  $\I{p_1}{\times}\I{p_1}$ matrix $\overline{\kappa}^1(y)$ as follows:  For $i$, $j{\in}\I{p_1}^*$, $j{\ne}i$,   $\overline{\kappa}^1_{ij}(y){=}\kappa_{ij}$ and 
\[
\overline{\kappa}^1_{0i}(y)= \kappa_{0i}+\sum_{j\in \I{2,p_2}^*} y_j^{k_j}\kappa_{ji}, \quad 
\overline{\kappa}^1_{i0}(y)=\kappa_{i0}+\sum_{j\in \I{2,p_2}^*} \kappa_{ij}.
\]
Remark that, for $i{\in}\I{p_1}$,
\[
\overline{\kappa}^{1,+}_i(y)=\sum_{j{\in}\I{p_1}\setminus\{i\}}\overline{\kappa}^1_{ij}(y)=\kappa_1^+.
\]
It is easily seen that $\overline{\kappa}^1{\in}\Omega(\I{p_1})$ and 
\[
K_i[y](z)=\overline{\kappa}_{0i}^1(y) +\sum_{j\in \I{p_1}^*\setminus\{i\}}z_j\overline{\kappa}_{ji}^1(y) -\overline{\kappa}_{i}^{1,+}(y) z_i.
\]
Note that if $\I{2, p_2}^*$ is empty, then $\overline{\kappa}^1$ is then constant, there is no dependence on $y$ of course, and Theorem~\ref{Thk2} is proved for $m_0=1$. 

Now if $\I{2,p_2}^*$ is not empty, for  $y{\in}{\cal K}_{\I{2,p_2}}$, the equation $$K_1[y](z^p){=}0$$ is  the system~\eqref{system2} of Proposition~\ref{InvK} for the set of indices $\I{p_1}$ and the matrix $\overline{\kappa}^1(y)$. It has a unique solution  $z{=}\widetilde{L}_{1}[y]{=}(L^{1}_i(y),i{\in}\I{p_1}^*)$. 
We now define an \emph{entropy function} $H_1$   given by, for $y{\in}{\cal K}_{\I{2,p_2}}$ and $z{\in}{\cal K}_{\I{p_1}}$, 
\begin{equation}\label{defV1}
		H_1[y](z)= \sum_{i\in \I{p_1}^*} z_i\ln\left( \frac{z_i}{L^{1}_i(y)^{p_1}}\right){-}z_i.
\end{equation}
Note that  $H_1[y]$ is a $\cal{C}^2${-}function on ${\cal K}_I$.
It is easily checked that  Relation~\eqref{weq1} for the function $f{:}x{\mapsto}H_1[x_{[2,p_2]}](x_{[p_1]})$ can be rewritten as \begin{equation}\label{xeq1}
\int_0^{T} \int_{{\cal K}_I}   F_1[x_{[2,p_2]}](x_{[p_1]}{}^{p_1})\pi_s(\diff x)\diff s=0,
\end{equation}
where, for $z{\in}{\cal K}_{\I{p_1}}$, 
\[
F_1[y](z)\steq{def} \sum_{i\in \I{p_1}^*} \left(z_i\overline{\kappa}_{i}^{1,+}(y)-\overline{\kappa}^1_{0i}(y)-\sum_{j\in \I{p_1}^*\setminus\{i\}} z_j\overline{\kappa}_{ji}(y)\right) \ln \left(\frac{z_i}{L^1_{i}(y)^{p_1}}\right).
\]
Note that, for  $y{\in}{\cal K}_{\I{2,p_2}}$,  $F_1[y]$ is the function $F_{\overline{\kappa}^1(y)}$ of Relation~\eqref{DefF} for the set of indices $\I{p_1}$. 
Relation~\eqref{xeq1} gives therefore that, almost surely, 
\[
\int_0^{T} \int_{y{\in}{\cal K}_{\I{2,p_2}}} \left(\int_{z{\in}{\cal K}_{\I{p_1}}}F_1[y](z^{p_1})\pi_s^{[p_1]}(\diff z|y)\right)\diff s{\otimes}\pi_s^{[2,p_2]}(\diff y)=0,
\]
with the notations of Section~\ref{Not} and, for $s{\ge}0$,  $\pi_s^{[p_1]}(\diff z|y)$ is the conditional distribution on ${\cal K}_{\I{p_1}}$ of $\pi_s{\in}{\cal P}(\R_+^{I^*})$ with respect to $y{\in}{\cal K}_{\I{2,p_2}}$. 
Consequently, since $F_1[y]$ is non-negative, up to a negligible set of $[0,T]{\times}{\cal K}_{\I{2,p_2}}$ for the measure  $\diff s{\otimes}\pi_s^{[2,p_2]}(\diff y)$, we have the relation 
\[
\int_{{\cal K}_{\I{p_1}}}F_1[y](z^{p_1})\pi_s^{[p_1]}(\diff z|y)=0.
\]
Proposition~\ref{PropConvex} gives that $\widetilde{L}_{1}(y)$ is the only root of the function $x{\mapsto}F_1[y](x^{p_1})$ on ${\cal K}_{\I{p_1}}$, hence 
the probability distribution $\pi_s^{[p_1]}(\diff x|y)$ is the Dirac measure at $\widetilde{L}_{1}(y)$. 

If  $h$, $f_1$ and $f_2$, are continuous functions on, respectively, $[0,T]$, ${\cal K}_{\I{2, p_2}}$ and ${\cal K}_{\I{p_1}}$ then,  almost surely, 
\begin{align*}
\int_0^{T} \int_{{\cal K}_{I}}&h(s)f_1(x_{[2,p_2]})f_2(x_{[p_1]})\pi_s(\diff x)\diff s\\=&\int_0^{T}\int_{y{\in}{\cal K}_{\I{2,p_2}}}f_1(y) \int_{z{\in}{\cal K}_{\I{p_1}}}h(s)f_2(z)\pi_s^{[p_1]}(\diff z|y)\pi_s^{[2,p_2]}(\diff y)\diff s\\
=& \int_0^{T}\int_{y{\in}{\cal K}_{\I{2,p_2}}}h(s)f_1(y)f_2(\widetilde{L}_{1}(y))\pi_s^{[2,p_2]}(\diff y)\diff s.
\end{align*}
We get therefore that for $f{\in}{\cal C}_c([0,T]{\times}{\cal K}_I)$, almost surely,
\begin{equation}\label{a1eq}
\int_0^T \int_{{\cal K}_I}f(s,x)\pi_s(\diff x)\diff s  =\int_0^T \int_{{\cal K}_{\I{2,p_2}}}f\left(s,(y,\widetilde{L}_{1}(y))\right)\pi_s^{[2,p_2]}(\diff y)\diff s,
\end{equation}
with the slight abuse of notation of writing $x{=}(x_{[2,p_2]},x_{[p_1]})$ for $x{\in}\R_+^{I^*}$. 

We can now use our induction  assumption to identify the measure $\diff s\otimes \pi_s^{[2,p_2]}(\diff y)$. 
To do so, we have to show that a set of equations as in Relation~\eqref{eqLimitMeas} for $\pi_s^{[2,p_2]}$ and an appropriate $\overline{\kappa}^2$. 

If we can find some $\overline{\kappa}^2\in \Omega(\I{2, p_2})$ depending only on the initial $\kappa$ such that for all $y\in {\cal K}_{\I{2,p_2}}$, for all $i\in \I{2, p_2}^*$, 
\begin{multline}\label{FoudKappa2}
	\kappa_{0i}+\sum_{j\in \I{p_1}^*} \kappa_{ji} (L^{1}_i(y))^{p_1} +\sum_{j\in \I{2,p_2}^*\setminus\{i\}} \kappa_{ji} y_j^{k_j} -\kappa_{i}^+y_i^{k_i} \\
		=\overline{\kappa}^{2}_{0i} +\sum_{j\in \I{2,p_2}^*\setminus\{i\}} \overline{\kappa}^{2}_{ji} y_j^{k_j} -\overline{\kappa}^{2,+}_{i} y_i^{k_i}.
\end{multline}
Applying Relation~\eqref{a1eq} in Relation~\eqref{eqLimitMeas}, for any $2{\leq} p{\leq}p_2$, for any $f{\in}{\cal C}^2\left({\cal K}_{\I{2,p}}\right)$, almost surely, we have that  the relation 
\[
	\int_0^t \int_{{\cal K}_{\I{2,p_2}}}\sum_{i\in \I{p}^*} \left(\overline{\kappa}^{2}_{0i}{+}\sum_{j\in I^*{\setminus}\{i\}}\overline{\kappa}^{2}_{ji} x_j^{k_j} {-}\overline{\kappa}^{2,+}_i x_i^p  \right) \frac{\partial f}{\partial x_i}(x_{[2,p]}) \pi^{[2, p_2]}_s(\diff x)\diff s =0,  
\]
holds for $t\in[0, T]$. We recognize here the Relations of Proposition~\ref{propEqLimit}, for the set of indices $\I{2,p_2}$ and the matrix $\overline{\kappa}_{2}{\in} \Omega(\I{2,p_2})$. We can apply the induction hypothesis on the measure $\pi^{[2, p_2]}$. 
Setting $\widetilde{L}_{2}$ the unique solution of the system~\eqref{system2} of Proposition~\ref{InvK} for the set of indices $\I{2, p_2}$ and the matrix $\overline{\kappa}^2$, Relation \eqref{a1eq} can be rewritten as : for $f{\in}{\cal C}_c([0,T]{\times}{\cal K}_I)$, almost surely,
\begin{equation}
\int_0^T \int_{{\cal K}_I}f(s,x)\pi_s(\diff x)\diff s  =\int_0^T \int_{{\cal K}_{\I{2,p_2}}}f\left(s,(\widetilde{L}_2,\widetilde{L}_{1}(\widetilde{L}_2))\right)\diff s,
\end{equation}
with the slight abuse of notation of writing $x{=}(x_{[2,p_2]},x_{[p_1]})$ for $x{\in}\R_+^{I^*}$. 

We conclude the induction by checking that
\[
(\widetilde{L}_2,\widetilde{L}_{1}(\widetilde{L}_2))=\ell_{\kappa},
\]
where $\ell_{\kappa}{=}(\ell_{\kappa,i})$ is the unique solution of the system~\eqref{system2} of Proposition~\ref{InvK}.

For the existence of $\overline{\kappa}^2$ that verifies Relation~\eqref{FoudKappa2}. It is done by induction on the number of elements of the set $\I{p_1}^*$. If this set contains only one index $i_0$, setting $\overline{\kappa}^{i_0}$ such that  for $i$, $j{\in}\I{1, p_2}$, $j{\ne}i$,   
\begin{equation}\label{eqDefKappa0}
\overline{\kappa}^{i_0}_{ij}= \kappa_{ij}+\frac{\kappa_{ii_0}\kappa_{i_0j}}{\kappa_{i_0}^+}, 
\end{equation}
is suitable. Otherwise, if $\I{p_1}^*$ contains more than one element, we remove them, one by one, by applying the transformation of Relation~\eqref{eqDefKappa0}.

The theorem is proved.

\end{proof}
\section{The General Case}\label{1GenSec}
We can now conclude the proof of Theorem~\ref{Theorem}. The difference with Section~\ref{FastSec} is the time-inhomogeneity of the limiting quantities. 
\begin{proposition}\label{PropTightGen}
  If the initial conditions satisfy Relation~\eqref{CondIni} then the sequence of processes $((X^N_{[1]}(t)),\Lambda_N)$, defined by Relations~\eqref{ScaledProcess}    and~\eqref{OccMeas},  is tight for the convergence in distribution.  Any limiting point $((x(t)),\Lambda_\infty)$ is such that
  \begin{enumerate}
  \item Almost surely, $(x(t))$ is a continuous process with values in ${\cal K}_{\I{1}}$;
  \item For any function $f{\in}{\cal C}_c([0, T]{\times}(\R_+)^{I^*})$,
\begin{equation}\label{ConvKurtzkgeq3}
\croc{\Lambda_\infty, f}=\int_{[0, T]\times {\cal K}_{\I{2+}}}f\left(s, \left(x(s),y\right)\right) \pi^{[2+]}_s(\diff y) \diff s, 
\end{equation}	
where $(\pi^{[2+]}_s)$ is an optional process with values in  ${\cal P}({\cal K}_{\I{2+}})$.
  \end{enumerate}
\end{proposition}
Recall the convention of writing an element $x$ of $(\R_+)^{I^*}$ as $x{=}(x_{[1]},x_{[2+]})$. See Section~\ref{Not}. 
\begin{proof}
  The tightness of the occupation measures is shown exactly as in the proof of Proposition~\ref{PropTightkgeq2}. Definition~\eqref{HNTN}, Theorem~\ref{UnTh} shows that the tightness of $(\overline{X}_N(t{\wedge}T_N))$ gives the tightness of the sequence of processes  $(\overline{X}_N(t))$. It  is established via the criterion of the modulus of continuity. See Theorem~7.3 of~\citet{Billingsley}.

For $i\in \I{1}^*$,  $\delta{>}0$, Relation~\eqref{SDE} gives the relation
\begin{multline*}
w_i^N(\delta)\steq{def} \sup_{\substack{s,t\le T{\wedge}T_N\\|s-t|\le \delta}} \left|\overline{X}_i^N(t){-}\overline{X}_i^N(s)\right|\le \kappa_{0i}\delta{+}2\sup_{t\le T{\wedge}T_N} |\overline{M}_N(t)|\\+\sum_{j\in I^*{\setminus}\{i\}}  \kappa_{ji}\int_s^t \overline{X}_j^N(u)^{(k_j)}\diff u
+\sum_{j\in I{\setminus}\{i\}} \kappa_{ij}\int_s^t \overline{X}_i^N(u)\diff u,
\end{multline*}
where $(\overline{M}_N(t{\wedge}T_N))$ is a martingale whose previsible increasing process at time $T$ is
\[
\frac{ k_i^2}{N} \left(\kappa_{0i}T{\wedge}T_N{+}\hspace{-3mm}\sum_{j\in I^*{\setminus}\{i\}} \kappa_{ji}\int_0^{T{\wedge}T_N} \overline{X}_j^N(u)^{(k_j)}\diff u
{+}\hspace{-3mm}\sum_{j\in I{\setminus}\{i\}}\kappa_{ij}\int_0^{T{\wedge}T_N}\overline{X}_i^N(u)^{(k_i)}\diff u\right).
\]
The expected value of this quantity on the event ${\cal E_N}$ converge to $0$, by Doob's Inequality and Theorem~\ref{UnTh}, the martingale $(\overline{M}_N(t{\wedge}T_N))$ converges in distribution to $0$. The proposition is proved. 
\end{proof}

\begin{proposition}\label{propEqLimitGeneral}
  If   $((x(t)),\Lambda_\infty)$ is a limiting point of  $((\overline{X}_N(t)),\Lambda_N)$ with the representation~\eqref{ConvKurtzkgeq3}, then for $p\geq 2$, for $f{\in}\cal{C}^2((\R_+^*)^{\I{2,p}^*})$, almost surely, for all $t{\in}[0,T]$,  the relation
\begin{multline}\label{eqLimitMeasGeneral}
                \int_0^t \int_{{\cal K}_{\I{2+}^*}}\sum_{i\in \I{p}^*} \left(\sum_{j\in \I{1}^*}\kappa_{ji}x_{j}(s)+\kappa_{0i}+\sum_{\substack{j\in \I{2+}^*\setminus\{i\}}}\kappa_{ji} y_j^{k_j} -\kappa_{i}^+ y_i^p  \right)\\
                \frac{\partial f}{\partial x_i}(y_{[2,p]}) \pi_s^{[2+]}(\diff y)\diff s =0. 
        \end{multline}
        holds. 
\end{proposition}
\begin{proof}
We take a subsequence $((\overline{X}_{[1]}^{N_p}(t)),\Lambda_{N_p})$ converging in distribution to the random variable  $((x(t)),\Lambda_\infty)$.  The occupation measure of $(\overline{X}_i^N(t), i{\in}\I{2+})$ is converging in distribution to $\Lambda_\infty^{[2+]}$ defined by
  \[
\croc{\Lambda_\infty^{[2+]},f}=\int_0^T\int_{{\cal K}_{\I{2+}}} g(y)\pi_s^{[2+]}(\diff y)\diff s,
\]
for $f{\in}{\cal C}_c((\R_+)^{\I{2+}})$. Since the process $(\overline{X}_i^{N_p}(t),i{\in}\I{1}^*)$ converges in distribution, for the uniform norm on $[0,T]$, we obtain a representation of $\Lambda_\infty$,
\begin{equation}\label{Heq}
\croc{\Lambda_\infty,g}=\int_0^T \int_{{\cal K}_{I}}g(y)\pi_s(\diff y)\diff s=\int_0^T\int_{{\cal K}_{\I{2+}}} g(x(s),y)\pi_s^{[2+]}(\diff y)\diff s,
\end{equation}
for $g{\in}{\cal C}_c((\R_+)^{I^*})$. With the same method as in the proof of Proposition~\ref{propEqLimit}, the analogue of Relation~\eqref{eqLimitMeas} is established. We conclude the proof by using Relation~\eqref{Heq}. 
\end{proof}

\begin{proof}[Proof of Theorem~\ref{Theorem}]
In view of Theorem~\ref{Thk2}, we can assume $\I{1}{\ne}\emptyset$. 

First, lets identify $\Lambda_\infty$. Using Relation~\eqref{Heq}, we only have to identify the measure $\diff s\otimes \pi_s^{[2+]}(\diff y)$ on $\R_+\times (\R_+^*)^{\I{2, +}^*}$. 

We set for $t\in [0,T]$, $\overline{\kappa}^3(t)$ as follows: For $i$, $j{\in}\I{2+}^*$, $j{\ne}i$,   $\overline{\kappa}^3_{ij}(t){=}\kappa_{ij}$ and 
\[
\overline{\kappa}^3_{0i}(t)= \kappa_{0i}+\sum_{j\in \I{1}^*} x_j(t)\kappa_{ji}, \quad 
\overline{\kappa}^3_{i0}(t)=\kappa_{i0}+\sum_{j\in \I{1}^*} \kappa_{ij}.
\]

Relation~\eqref{eqLimitMeasGeneral}, can be rewritten as Relation~\eqref{eqLimitMeas}, for the set of indices $\I{2+}$ and the matrix $(\overline{\kappa}^{3}(t))\in \Omega(\I{2,p_2})^{[0, T]}$. The species $i\in \I{1}^*$ behave for the fast species as \emph{exterior input}, with rate time dependent.

The convergence of the measure $\pi^{[2+]}(\diff y)$ is then shown similarly as in the proof of Theorem~\ref{Thk2}. The only difference is the time dependence of the $\overline{\kappa}^3(t)$, which does not create any difficulty, since Proposition~\ref{PropTightGen} gives the continuity of $(\overline{\kappa}^{3}(t))$ on $[0, T]$. Using Relation~\eqref{Heq}, for $g{\in}{\cal C}_c((\R_+)^{I^*})$, we have 
\[
	\croc{\Lambda_\infty,g}=\int_0^T \int_{{\cal K}_{I}}g(y)\pi_s(\diff y)\diff s=\int_0^T g((x(s)),\ell(s))\diff s,
\] 
where for all $t\in [0, T]$ $\ell(t)$ is the unique solution of the system~\eqref{system2} of Proposition~\ref{InvK} for the set of indices $\I{2+}$ and the matrix $\overline{\kappa}^3(t)$. It is easily seen that for all $t\in [0, T]$, $$\ell(t)=L(x(t)),$$ where $L$ is defined in Relation~\eqref{LinearSysEll}. 

The convergence of the occupation measure is shown. 

For the identification of the function $(x_{[1]}(t)){=}(x_i(t), i{\in} \I{1}^*)$, integrating Relation~\eqref{SDE} and dividing it by $N$, we get for $t{\in}[0, T]$, $i{\in}\I{1}^*$:
\begin{multline}\label{eqX1Integrale}
	\overline{X}_i^N(t) = \overline{X}^N_i(0)+ M_i^N(t) +\sum_{j\in \I{1}^*\setminus\{i\}} \int_0^t\kappa_{ji}\overline{X}_j^N(s)\diff s \\
	+\int_0^t\sum_{j\in \I{2+}} \kappa_{ji}\frac{(X^N_j(s))^{(k_i)}}{N}\diff s -\kappa_{i}^+ \int_0^t \overline{X}_i^N(s)\diff s,
\end{multline}
where $(M_i^N(t))$ is a local martingale whose previsible increasing process is given by, for $t\leq T$, 
\begin{multline*}
	\croc{M_i^N}(t) = \frac{1}{N}\sum_{j\in \I{1}^*\setminus\{i\}} \int_0^t\kappa_{ji}\overline{X}_j^N(s)\diff s \\
	+\frac{1}{N}\sum_{j\in \I{2+}}\int_0^t \kappa_{ji}\frac{(X^N_j(s))^{(k_i)}}{N}\diff s +\frac{\kappa_{i}^+}{N} \int_0^t \overline{X}_i^N(s)\diff s. 
\end{multline*}
Using Doob's inequality and the bound of $(\overline{X}^N(t))$ on the event $\cal{E}_N$, we get the convergence in distribution of the martingales to $0$. 

Relation \eqref{TechnicEqApprox}, Lemma~\ref{lemaux}, and the convergence of $\pi_s^{2+}$ just proven, shows that for the convergence in distribution, for $j\in \I{2+}^*$, 
\[
	\lim_{N\to +\infty} \left(\int_0^t \kappa_{ji}\frac{(X^N_j(s))^{(k_i)}}{N}\diff s, t\in [0, T]\right) = \left(\int_0^t \kappa_{ji}(L_j(x(s)))^{k_j}\diff s, t\in [0, T]\right), 
\]
and therefore, taking $N$ to infinity in Relation~\eqref{eqX1Integrale}, we get for $t{\in} [0, T]$, $i{\in} \I{1}^*$:
\[
	x_i(t)= \alpha_i+\sum_{j\in \I{1}^*\setminus\{i\}} \int_0^t\kappa_{ji}x_j(s)\diff s \\
	+\int_0^t\sum_{j\in \I{2+}} \kappa_{ji}(L_j(x(s)))^{k_j}\diff s -\kappa_{i}^+ \int_0^t x_i(s)\diff s, 
\]
which is exactly Relation~\eqref{ODE}.

Since $(x_{[1]}(t))$ lives in ${\cal K}_{\I{1}}$, the solution of this ODE is unique, and therefore the identification of $(x_{[1]}(t))$ is complete. 

\end{proof}

Note that ODE~\eqref{ODE} can be rewritten as 
\[
	\dot{x}_i(t)= \overline{\kappa}^4_{0i}+\sum_{j\in \I{1}^*\setminus\{i\}} x_j(t)\overline{\kappa}^4_{ji}-x_i(t) \sum_{j\in I\setminus\{i\}} \overline{\kappa}^4_{ij},\quad i{\in} \I{1}^*,
\]
where $\overline{\kappa}^4{\in}\Omega(\I{1})$ is a matrix depending on the initial $\kappa$, constructed following the steps of the construction of $\overline{\kappa}^2$ in the Proof of Theorem~\ref{Thk2}. The $\overline{\kappa}^4$ can be given explicitly in terms of  a path between complexes of $\I{1}^*$. The simplified ODE corresponds to the ODE associated to a CRN with only the complexes 
$$\{\emptyset\}{\cup}\{S_i, i{\in}\I{1}^*\},$$
whith reactions defined by $\overline{\kappa}^4$. 
As an example, the limit $(x_4(t))$ of $(\overline{X}^N_4(t))$ in the CRN of Figure~\ref{FourSpecies} is solution of the ODE associated to the CRN 
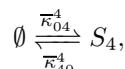
\[
	\emptyset \mathrel{\mathop{\xrightleftharpoons[{\overline{\kappa}^4_{40}}]{{\overline{\kappa}^4_{04}}}}} S_4, 
\]
with 
\[
	\overline{\kappa}^4_{04}=\frac{\kappa_{01}\kappa_{12}\kappa_{24}}{\kappa_{1}^+\kappa_{2}^+} +\frac{\kappa_{01}\kappa_{13}\kappa_{34}}{\kappa_{1}^+\kappa_{3}^+}+\frac{\kappa_{01}\kappa_{12}\kappa_{23}\kappa_{34}}{\kappa_{1}^+\kappa_{2}^+\kappa_{3}^+} \quad \text{and} \quad \overline{\kappa}^4_{40}=\frac{\kappa_{43}\kappa_{30}}{\kappa_{3}^+}.
\]

\printbibliography

\end{document}